\documentclass[a4paper]{article}%
\usepackage{amsmath}
\usepackage{amsfonts}
\usepackage{amssymb}
\usepackage{graphicx}%
\providecommand{\U}[1]{\protect\rule{.1in}{.1in}}
\newtheorem{theorem}{Theorem}
\newtheorem{definition}[theorem]{Definition}
\newtheorem{assumption}[theorem]{Assumption}

\newtheorem{remark}[theorem]{Remark}
\newenvironment{proof}[1][Proof]{\noindent\textbf{#1.} }{\ \rule{0.5em}{0.5em}}
\begin{document}

\title{Sensor Deployment for Network-like Environments}
\author{Luca~Greco\thanks{L. Greco and M. Gaeta are with the DIIMA, University of
Salerno, Via Ponte don Melillo 1, 84084 Fisciano (SA), Italy; e-mail:
\texttt{lgreco@ieee.org}, \texttt{gaeta@crmpa.unisa.it}}, Matteo Gaeta,
Benedetto Piccoli\thanks{B. Piccoli is with the Istituto per le Applicazioni
del Calcolo \textquotedblleft Mauro Picone\textquotedblright, Via dei Taurini
19, 00185 Roma, Italy; e-mail: \texttt{b.piccoli@iac.cnr.it}}}
\maketitle

\begin{abstract}
This paper considers the problem of optimally deploying omnidirectional
sensors, with potentially limited sensing radius, in a network-like
environment. This model provides a compact and effective description of
complex environments as well as a proper representation of road or river
networks. We present a two-step procedure based on a discrete-time gradient
ascent algorithm to find a local optimum for this problem. The first step
performs a coarse optimization where sensors are allowed to move in the plane,
to vary their sensing radius and to make use of a reduced model of the
environment called collapsed network. It is made up of a finite discrete set
of points, barycenters, produced by collapsing network edges. Sensors can be
also clustered to reduce the complexity of this phase. The sensors' positions
found in the first step are then projected on the network and used in the
second finer optimization, where sensors are constrained to move only on the
network. The second step can be performed on-line, in a distributed fashion,
by sensors moving in the real environment, and can make use of the full
network as well as of the collapsed one. The adoption of a less constrained
initial optimization has the merit of reducing the negative impact of the
presence of a large number of local optima.

The effectiveness of the presented procedure is illustrated by a simulated
deployment problem in an airport environment.

\end{abstract}

\section{Introduction}

Imagine a scenario where a toxic gas is spreading in an area or a building and
safe paths have to be found to evacuate people. Or think of an airport
environment where people moving through rooms and corridors has to be
surveilled in order to detect and avoid terroristic actions. Or consider the
need of measuring environmental quantities, such as temperature or humidity,
on wide areas to the aim of improving theoretical models or making more
accurate weather forecast.

There is a great number of situations that would greatly enjoy the use of
network of\ sensors. Indeed, many of the previous tasks are difficult, or
impossible, to be accomplished by a single sensor. The employment of a large
number of sensors increases the robustness to sensor failure and communication
disruption and make spatially-distributed observations possible. If sensors
are able to move, the number of tasks they can perform is still greater.

Static and dynamic sensors' networks need to be deployed in the environment,
and the way this problem is solved can significantly affect the quality of
service they have to provide.

\subsection{Static Deployment and Locational Optimization}

Sensors' deployment problems are strictly related to facilities
location-allocation problems, which are the subject of the locational
optimization discipline (\cite{DreznerBOOK95}).

In locational optimization objective functions are used to describe the
interactions between users and facilities and among them. Users may find
facilities desirable, hence they would like to exert an attractive force to
facilities, or undesirable and they would repel them. The attractive model can
describe allocation problems of useful services or facilities such as
mailboxes, hospitals, fire stations, malls, etc. (see \cite{DreznerBOOK95}).
The repulsive one, instead, can be used to model problems where polluting or
dangerous facilities (i.e. nuclear reactors, garbage dumps, etc.) are to be
located far enough from urban conglomerations. An excellent survey on
undesirable facility locations problems is given by \cite{ErkutEJOR89} (see
also \cite{CappaneraTR99}). These operational research problems can be
converted in sensors' deployment problems by considering sensors as facilities
and points or areas, where events can happen or some quantities has to be
measured, as users.

Two well known problems, involving one facility only, are the classical Weber
and the obnoxious facility location problems (see
\cite{DreznerOR04,DreznerSWDSIAC07} for a recent heuristic solution). Three
problems involving $p$ facilities are the $p$-center, $p$-median and
$p$-dispersion problems. Some recent results on the $p$-center problem are in
\cite{SuzukiLS96} and \cite{CortesJCO05}. The latter paper addresses also the
$p$-dispersion problem.

A classical $p$-median problem, close to the one considered in this paper, can
be simply described as the one of finding the optimal location of $p$
facilities by minimizing the average distance of the demand points to the
nearest facility (see \cite{MladenovicEJOR07} for a recent survey on
heuristics methods to solve it). Close to the $p$-median problem is the
multisource Weber problem, for which many heuristics exist
(\cite{BrimbergOR00}). A more general formulation of these problems can be
found in \cite{CortesTRA04} and \cite{DuSIAMREV99}. In \cite{CortesTRA04} a
dynamical (gradient descent) version of the Lloyd's algorithm
\cite{LloydTIT82} has been presented to find a local optimum for a generalized
$p$-median problem. A deterministic annealing optimization algorithm to solve
the classical version is reported in \cite{SharmaACC06}. The aforementioned
solutions to the $p$-median problem, as well as many solutions to
$p$-facilities problems, are based on the construction of a Voronoi
Tesselation (\cite{DuSIAMREV99,OkabeEJOR97}).

\subsection{Dynamic Deployment and Distributed Solutions}

The use of moving, instead of static, sensor networks provide a great
flexibility in solving sensing tasks, mainly when the environment is partially
or completely unknown or is not directly accessible for safety reasons. In
these cases, sensors are usually initially deployed randomly and hence need to
move in order to acquire knowledge of the environment and to optimally
re-deploy for their task. Furthermore, environments are usually not static and
the network may experience sensor failure or loss. In these situations the
properties of adaptivity and reconfigurability owned by a network of moving
sensors turn out very useful.

A general tendency in robotic networks is to have sensors (agents) endowed
with the same computational and sensing capabilities. This choice increases
the overall robustness of the network, but usually calls for distributed
coordination algorithms. Having equal sensors, indeed, naturally leads to
define optimization and coordination algorithms based on local observations
and local decisions (\cite{GanguliCDC05,MartinezMag07,JadbabaieTAC03}). Many
of the algorithms proposed in the previous section involve the solution of a
global optimization problem requiring a complete knowledge of the environment
and of sensors' distribution. The solutions to $p$-center, $p$-dispersion and
$p$-median problems proposed in \cite{CortesJCO05,CortesTRA04,CortesESAIM05},
instead, are all spatially distributed, with the meaning that each sensor
requires only the knowledge of positions of its neighbors (or even less if it
has a limited sensing radius). This fact allows a distributed implementation
where each sensor computes its next movement without centralized coordination.

Other solutions to the area-coverage problem look at sensors like particles
subject to virtual forces or potential fields. The compositions of suitably
defined attractive and repulsive forces is then used to make the network
behaving in the desired fashion (spread sensors, avoid obstacles, keep
connectivity, etc.). Representative for this kind of approach are the
algorithms presented in \cite{HowardDARS02,ZouTECS04}, or in \cite{MaoSASN05},
where also secure connectivity issues are considered. In \cite{HeoTSMC05},
instead, it is raised the relevant problem of power consumption in wireless
networks and three energy-efficient algorithms are presented for sensors' deployment.

\subsection{Network-like Environments and Paper Contributions}

In this paper, we focus on network-like environments as there are surveillance
or monitoring problems where such a kind of model can provide a more suitable description.

Network models represent a natural choice whenever environments have an
intrinsic network structure. It is the case, for instance, when sensors have
to be optimally deployed over a network of roads to monitor vehicular traffic,
or in a river network to measure temperature or pollutants concentration. Even
some location-allocation problems can involve networks. Consider, for
instance, the case where useful facilities (i.e. schools, hospitals) have to
be located in the interior of a network of roads, which is the source of a
nuisance (i.e. noise, pollutants), with the goal of minimizing its harmful
impact on them (see \cite{DreznerSWDSIAC07} for the case with one facility);
or the dual problem of locating obnoxious facilities (i.e. dumps, industrial
plants, mobile phone repeater antennas) reducing the hazard on the network.

Most notably, we think that a network-like model can provide an effective and
compact description for complex environments, focusing only on major features
and abstracting from those geometrical details which are less important for
the deployment problem. The coverage of nonconvex environments with holes or
obstacles, for instance, is a challenging task (\cite{MilosCCA08}), which can
enjoy significantly the use of a reduced network-like model. Environments with
a complex structure, accounting for a large number of variously sized and
shaped rooms, passages, forbidden areas and obstacles, can be reduced to a set
of connected paths where the sensing task is more requested or where sensors
are forced to pass. An airport is a very representative example of this kind
of environments. In this case people moving throughout the airport can be
aptly compared with a network flow and focus can be on paths more than on
corridors, halls and lounges.

Many of the problems introduced in previous sections have been formulated even
for a network-like environment (\cite{TanselMS83pI,TanselMS83pII}), but they
usually consider a finite discrete set of demand points located on the
network's nodes and try to optimize the locations of facilities w.r.t. some
objective function accounting for the distance from them. An important thread
of works for the deployment of sensors both on plane and on a network is
represented by the papers \cite{Krause06,Krause08,Krause09}. In such works,
the authors propose powerful greedy algorithms that provide a constant
approximation of the optimal solution. Their method, however, aims at solving
a global static deployment problem, considering a finite set of demand points
and allowing sensors to jump among positions.

In this paper, instead, we consider a deployment problem where sensors use
local information to dynamically solve the optimization problem while they are
moving in the environment. The constraints induced by the environment to the
motion of sensors are explicitly considered. More precisely, we address a
generalized $p$-median problem involving omnidirectional sensors with
potentially limited sensing radius and we extend the formulation presented in
\cite{CortesESAIM05,CortesTRA04,MartinezMag07} to network-like environments.
The task is to find sensors' positions that optimize an objective function
defined on the network and accounting for the sensor's features and
preferential areas. This is a mixed problem, since the network is considered
embedded in the plane (it is a continuous set of demand points) and the planar
euclidean norm is used to measure the distance between sensors and network.

The core of the cited formulation and of our solution is a discrete-time
gradient ascent algorithm based on Voronoi partitions and aiming at maximizing
the objective function. It is a well known fact, however, that such a kind of
algorithms can get stuck early in local optima, especially when sensors are
forced to move in an over-constrained environment like a network. Moreover,
the local maximum found by the algorithm is often greatly related to the
initial sensors' position.

For these reasons, we present a novel two-step procedure performing an initial
coarse optimization, whose purpose is to provide a good starting point for a
second finer optimization. The first step can be carried out off-line, either
by a central unit, or by each sensor individually (without doing real
movements). The impact of local optima is reduced by allowing sensors, in the
initial optimization, to virtually move in $\mathbb{R}^{2}$ and to vary their
sensing radius arbitrarily. In order to reduce the complexity of this phase,
sensors can be initially clustered and the optimization problem solved for the
clusters' centers. After that, a desired number of sensors is spread close to
clusters' centers and the sensors' positions thus found are projected on the
closest edges of the network. The projected positions are then used in the
second optimization, where sensors are constrained to move only on the
network. The second step can be performed on-line, in a distributed fashion,
by sensors moving in the real environment.

The use of a two-step procedure is motivated also by the very nature of some
surveillance tasks, such as for instance airport surveillance, where a large
number of individuals (sometimes referred to as mass objects \cite{GuoTSMC08})
are monitored. In these cases, sensors can solve the first step optimization
using imprecise or estimated information and keeping still; then they can move
to reach the final projected positions using planned routes compatible with
the network. After this initial deployment, sensors can change their positions
by dynamically solving a distributed optimization problem (second step) based
on real measures taken from the (potentially varying) environment.

It is worth noting that the present procedure can be used to solve both static
and dynamic deployment problems. Moreover, the first step deserves attention
by its own, since it provides a solution to those problems involving
facilities located in the interior of the network mentioned at the beginning
of this section.

Another contribution of this paper is the introduction of a simplified model
of the network (similar to the discretization in \cite{GaoAUT08}) called
\emph{collapsed network }and consisting of finite many points. It is obtained
by decomposing each segment of the original network in one or more
sub-segments and collapsing each sub-segment in its barycenter. This model
allows a coarser but faster optimization, since computations with barycenters
are remarkably less than those needed by the full network. Collapsed network,
hence, is intended mainly for fastening the first optimization, but can be
used profitably also for the second step. Indeed, it turns out particularly
useful in practical implementation involving hardware with limited
computational capabilities.

As mentioned above, our work is related to that of
\cite{CortesESAIM05,CortesTRA04,MartinezMag07}. In particular, the first step
of the optimization, allowing sensors to move in $\mathbb{R}^{2}$, could be
regarded as a specialization of the problem described in \cite{CortesTRA04}.
However, the different topology induced by the network introduces issues
related to the explicit computation of the gradient and to the convergence of
the maximization algorithm, which deserve special solutions. A relevant
difference is that the gradient of the objective function presents
discontinuity points caused by barycenters on the boundary edge of two
neighboring Voronoi cells. Such barycenters can change allocation during
sensors' motion, inducing abrupt variations to the value of the objective
function associated to each cell. This fact prevents a classical convergence
proof for gradient algorithms, hence we consider our proof as a minor
contribution of the paper. Some results about convergence may alternatively be
derived by using the method of Kushner and Borkar
(\cite{KushnerBook03,BorkarBook08}) of stochastic approximation to deal with
our differential inclusion.

The outline of the paper is as follows. In section \ref{sec:prelim} the
mathematical definitions of sensors, network and Voronoi covering are
introduced along with the objective function to be maximized to solve the
deployment problem. Section \ref{sec:collapsed} is devoted to the introduction
of the collapsed network and to the formulation, and proof of convergence, of
a gradient ascent algorithm to solve the first step (subsection
\ref{sec:SensinR2}) and the second step (subsection \ref{sec:sensonN}) of the
optimization procedure. Section \ref{sec:FullNet} addresses the solution of
the second step optimization involving the full network. Finally, a network
describing an airport environment is used in section \ref{sec:simul} to
illustrate by simulations the effectiveness of the proposed optimization
procedure. Conclusions and future research directions are reported in section
\ref{sec:conclusions}.

\section{Preliminaries and Problem Formulation\label{sec:prelim}}

In this section we introduce the mathematical framework to describe the
sensors, the network and its Voronoi covering.

\begin{definition}
Given two points $p_{1},p_{2}\in\mathbb{R}^{2}$, with $p_{1}\neq p_{2}$,
$s_{12}=[p_{1},p_{2}]\subset\mathbb{R}^{2}$ is the segment joining $p_{1}$ and
$p_{2}$ and $s_{12}^{o}=(p_{1},p_{2})$ is the open segment between them. We
define length of a segment $s_{12}$ as $\ell(s_{12})=\left\Vert p_{2}%
-p_{1}\right\Vert $, where $\left\Vert \cdot\right\Vert $ is the Euclidean
norm; barycenter of a segment $s_{12}$ the point $b(s_{12})=\frac{1}{2}\left(
p_{2}+p_{1}\right)  \in$ $s_{12}$; partition of a segment $s=[p_{1},p_{2}]$ in
$k$ sub-segments, the set of segments $\left\{  s_{i}\right\}  _{i=1,\ldots
,k}$ given by%
\[
s_{i}=\left[  p_{1}+\left(  i-1\right)  \frac{\left(  p_{2}-p_{1}\right)  }%
{k},p_{1}+i\frac{\left(  p_{2}-p_{1}\right)  }{k}\right]  \text{.}%
\]

\end{definition}

\begin{definition}
A network $\mathcal{N}=(\mathcal{V},\mathcal{S})$ is a subset of
$\mathbb{R}^{2}$ consisting of a set of points $\mathcal{V}=\left\{
v_{1},\ldots,v_{n}\in\mathbb{R}^{2},~v_{i}\neq v_{j}~\forall i\neq j\right\}
$ and a set of segments $\mathcal{S}\subseteq\{s_{ij}=[v_{i},v_{j}%
]\subset\mathbb{R}^{2},~i,j\in\left\{  1,\ldots,n\right\}  ~i\neq j\}$, such that:

\begin{description}
\item[i)] $\forall v_{i}\in\mathcal{V}$, $\exists v_{j}\in\mathcal{V}$,
$v_{i}\neq v_{j}$ such that $s_{ij}\in\mathcal{S}$ (no isolated vertex);

\item[ii)] $\forall i,j,h,k\in\left\{  1,\ldots,n\right\}  ,$ $(i,j)\neq
(h,k),$ $s_{ij}^{o}\cap s_{hk}^{o}=\emptyset$ (no segment intersection).
\end{description}
\end{definition}

\begin{definition}
Given a network $\mathcal{N}$ and a set of points $\mathcal{P}=\left\{
p_{1},\ldots,p_{m}\right\}  \subset\mathcal{N}$, the Voronoi covering of
$\mathcal{N}$ generated by $\mathcal{P}$ with respect to the Euclidean norm is
the collection of sets $\left\{  V_{i}^{\mathcal{N}}(\mathcal{P})\right\}
_{i\in\left\{  1,\ldots,m\right\}  }$ defined by%
\[
V_{i}^{\mathcal{N}}(\mathcal{P})=\left\{  q\in\mathcal{N}\mid\left\Vert
q-p_{i}\right\Vert \leq\left\Vert q-p_{j}\right\Vert ,\text{~}\forall p_{j}%
\in\mathcal{P}\right\}  \text{.}%
\]

\end{definition}

\begin{remark}
\label{rem:VoronoiCovering}It is straightforward to recognize that
$V_{i}^{\mathcal{N}}(\mathcal{P})$ can be equivalently defined as
$V_{i}^{\mathcal{N}}(\mathcal{P})=V_{i}(\mathcal{P})\cap\mathcal{N}$, where
$V_{i}(\mathcal{P})$ is the $i$-th cell of the usual Voronoi partition of
$\mathbb{R}^{2}$ generated by $\mathcal{P}$. The previous definition is about
a covering and not a partition since neighboring cells can have a nontrivial
intersection: a portion of a segment can belong to the shared edge of two
cells $V_{i}(\mathcal{P})$ and $V_{j}(\mathcal{P})$.
\end{remark}

We adapt the framework provided in \cite{CortesESAIM05} to describe the
sensors' and network features. Each sensor is modeled by the (same)
\emph{performance function} $f:\mathbb{R}_{+}\rightarrow\mathbb{R}$, that is a
non-increasing and piecewise differentiable map having a finite number of
bounded discontinuities at $R_{1},\ldots,R_{N}\in\mathbb{R}_{+}$, with
$R_{1}<\ldots<R_{N}$. We can set $R_{0}=0$, $R_{N+1}=+\infty$ and write%
\begin{equation}
f(x)=\sum_{\alpha=1}^{N+1}f_{\alpha}(x)1_{[R_{\alpha-1},R_{\alpha}%
)}(x)\text{,} \label{eq:def_f}%
\end{equation}
with $f_{\alpha}:[R_{\alpha-1},R_{\alpha}]\rightarrow\mathbb{R}$, $\alpha
\in\{1,\ldots,N+1\}$ non-increasing continuously differentiable functions such
that $f_{\alpha}(R_{\alpha})>f_{\alpha+1}(R_{\alpha})$ for $\alpha
\in\{1,\ldots,N\}$. In order to model regions of the network with different
importance, we can use a \emph{density function} $\phi:\mathcal{N}%
\rightarrow\mathbb{R}_{+}$, which is bounded and measurable on $\mathcal{N}$.
Given $g:\mathbb{R}^{2}\rightarrow\mathbb{R}$ we indicate by $%
{\displaystyle\int\nolimits_{\mathcal{N}}}
g(q)dq$ (respectively $%
{\displaystyle\int\nolimits_{V_{i}^{\mathcal{N}}(\mathcal{P})}}
g(q)dq$) the sum of the linear integrals of $g$ over the segments of
$\mathcal{N}$\ (respectively $V_{i}^{\mathcal{N}}(\mathcal{P})$) using an
arc-length parameterization. With these functions we can define the
multi-center function $\mathcal{H}:\mathcal{N}^{m}\rightarrow\mathbb{R}$ for
$m$ sensors located in $\mathcal{P}=\left\{  p_{1},\ldots,p_{m}\right\}
\subset\mathcal{N}$%
\begin{equation}
\mathcal{H}\left(  \mathcal{P}\right)  =%
{\displaystyle\int\nolimits_{\mathcal{N}}}
\max_{i\in\left\{  1,\ldots,m\right\}  }f\left(  \left\Vert q-p_{i}\right\Vert
\right)  \phi\left(  q\right)  dq\text{.} \label{eq:defH}%
\end{equation}

We can also provide an alternative expression for (\ref{eq:defH}) based on the
Voronoi covering induced by $\mathcal{P}$ as follows%
\begin{equation}
\mathcal{H}(\mathcal{P})=%
{\displaystyle\sum\limits_{i=1}^{m}}
{\displaystyle\int\nolimits_{V_{i}^{\mathcal{N}}(\mathcal{P})}}
f\left(  \left\Vert q-p_{i}\right\Vert \right)  \phi\left(  q\right)  dq-%
{\displaystyle\sum\limits_{\Delta_{hk}^{\mathcal{N}}\in\Delta^{\mathcal{N}}}}
{\displaystyle\int\nolimits_{\Delta_{hk}^{\mathcal{N}}}}
f\left(  \left\Vert q-p_{h}\right\Vert \right)  \phi\left(  q\right)
dq\text{,} \label{eq:defH_Vor}%
\end{equation}
where $\Delta_{hk}^{\mathcal{N}}\triangleq V_{h}^{\mathcal{N}}(\mathcal{P}%
)\cap V_{k}^{\mathcal{N}}(\mathcal{P})$ and $\Delta^{\mathcal{N}}%
\triangleq\left\{  \Delta_{hk}^{\mathcal{N}}\mid h<k,~\forall h,k\in\left\{
1,\ldots,m\right\}  \right\}  $. The second term in (\ref{eq:defH_Vor}) is not
null if and only if there exists a non trivial segment $s\subseteq s_{ij}%
\in\mathcal{S}$ such that $s\subset\Delta_{hk}^{\mathcal{N}}$ for some
$i,j\in\left\{  1,\ldots,n\right\}  $ and $h,k\in\left\{  1,\ldots,m\right\}
$.

\section{Deployment over a Collapsed Network\label{sec:collapsed}}

In a collapsed network each segment of the original network is decomposed in
one or more sub-segments and each sub-segment is collapsed in its barycenter.
Chosen a value for $r$ guaranteeing a good approximation, we can build the
$r$-collapsed network $\mathcal{C}_{r}^{\mathcal{N}}$ as follows:

\begin{definition}
[$r$-Collapsed Network]\label{def:collnet}Given a network $\mathcal{N}%
=(\mathcal{V},\mathcal{S})$ and $r>0$, $\forall s\in\mathcal{S}$ consider its
partition in $k_{s}=\left\lceil \frac{\ell(s)}{r}\right\rceil $ sub-segments
$s_{i}$ (having at most length $r$)\ and the associated set of barycenters
$\left\{  b\left(  s_{i}\right)  \right\}  _{i=1,\ldots,k_{s}}$. We define the
$r$-collapsed network associated to $\mathcal{N}$ the set of points
$\mathcal{C}_{r}^{\mathcal{N}}=%
{\displaystyle\bigcup\nolimits_{s\in\mathcal{S}}}
\left\{  b\left(  s_{i}\right)  \right\}  _{i=1,\ldots,k_{s}}$.
\end{definition}

The multi-center function must be re-defined since the integration domain is
now a discrete set represented by the barycenters. Hence we set%
\begin{equation}
\mathcal{H}(\mathcal{P})=%
{\displaystyle\sum\limits_{b_{e}\in\mathcal{C}_{r}^{\mathcal{N}}}}
\max_{i\in\left\{  1,\ldots,m\right\}  }f\left(  \left\Vert b_{e}%
-p_{i}\right\Vert \right)  \phi_{b_{e}}\text{,} \label{eq:defH_bar}%
\end{equation}
where $\phi_{b_{e}}$ are suitable (density) weights assigned to barycenters.

\subsection{Sensors Moving in $\mathbb{R}^{2}$\label{sec:SensinR2}}

In this section we solve the deployment problem for a collapsed network and
sensors moving in $\mathbb{R}^{2}$, which constitutes the first step of our
optimization procedure.

Also for the multi-center function (\ref{eq:defH_bar})\ we can provide an
alternative expression using the Voronoi covering. We need the following
definition\footnote{For this definition similar remarks as Remark
\ref{rem:VoronoiCovering} also apply.}:

\begin{definition}
\label{def:Voronoi_coll}Given an $r$-collapsed network $\mathcal{C}%
_{r}^{\mathcal{N}}$ for some $r\in\mathbb{R}_{+}$ and a set of points
$\mathcal{P}=\left\{  p_{1},\ldots,p_{m}\right\}  \subset\mathbb{R}^{2}$, the
Voronoi covering of $\mathcal{C}_{r}^{\mathcal{N}}$ generated by $\mathcal{P}$
with respect to the Euclidean norm is the collection of sets $\left\{
V_{i}^{\mathcal{C}_{r}^{\mathcal{N}}}(\mathcal{P})\right\}  _{i\in\left\{
1,\ldots,m\right\}  }$ defined by%
\[
V_{i}^{\mathcal{C}_{r}^{\mathcal{N}}}(\mathcal{P})=\left\{  b\in
\mathcal{C}_{r}^{\mathcal{N}}\mid\left\Vert b-p_{i}\right\Vert \leq\left\Vert
b-p_{j}\right\Vert ,\text{~}\forall p_{j}\in\mathcal{P}\right\}  \text{.}%
\]

\end{definition}

We define also the boundary of a Voronoi cell as
\[
\partial V_{i}^{\mathcal{C}_{r}^{\mathcal{N}}}(\mathcal{P})=\left\{
b\in\mathcal{C}_{r}^{\mathcal{N}}\mid\left\Vert b-p_{i}\right\Vert =\left\Vert
b-p_{j}\right\Vert ,\text{~}\forall p_{j}\in\mathcal{P}\right\}  \text{,}%
\]
and, in order to simplify the problem, we make the following assumption

\begin{assumption}
\label{Assumption1}$\partial V_{i}^{\mathcal{C}_{r}^{\mathcal{N}}}%
(\mathcal{P})=\emptyset$ $\forall i\in\left\{  1,\ldots,m\right\}  $.
\end{assumption}

With this assumption the multi-center function (\ref{eq:defH_bar}) can be
written also as%
\begin{equation}
\mathcal{H}(\mathcal{P})=%
{\displaystyle\sum\limits_{i=1}^{m}}
{\displaystyle\sum\limits_{b_{e}\in V_{i}^{\mathcal{C}_{r}^{\mathcal{N}}%
}(\mathcal{P})}}
f\left(  \left\Vert b_{e}-p_{i}\right\Vert \right)  \phi_{b_{e}}\text{,}
\label{eq:defH_bar_Vor}%
\end{equation}
or, setting $\mathrm{dist}(q,\mathcal{P})=\min_{i\in\left\{  1,\ldots
,m\right\}  }\left\Vert q-p_{i}\right\Vert $, as%
\begin{equation}
\mathcal{H}(\mathcal{P})=%
{\displaystyle\sum\limits_{b_{e}\in\mathcal{C}_{r}^{\mathcal{N}}}}
f\left(  \mathrm{dist}(b_{e},\mathcal{P})\right)  \phi_{b_{e}}\text{.}
\label{eq:defH_bar_dist}%
\end{equation}

\begin{remark}
\label{rem:Lip_coll}$\mathcal{H}(\mathcal{P})$ is not globally Lipschitz as
$f(\cdot)$ is not a continuous function. However, if $f(x)$ is continuous and
piecewise differentiable with bounded derivative, then $\mathcal{H}%
(\mathcal{P})$ is globally Lipschitz. In order to prove the global Lipschitz
continuity, let us consider two sets of points $\mathcal{P}=\left\{
p_{1},\ldots,p_{m}\right\}  \subset\mathbb{R}^{2}$ and $\mathcal{P}^{\prime
}=\left\{  p_{1}^{\prime},\ldots,p_{m}^{\prime}\right\}  \subset\mathbb{R}%
^{2}$ and compute%
\begin{align*}
\mathcal{H}(\mathcal{P})-\mathcal{H}(\mathcal{P}^{\prime})  &  =%
{\displaystyle\sum\limits_{b_{e}\in\mathcal{C}_{r}^{\mathcal{N}}}}
\left[  f\left(  \mathrm{dist}(b_{e},\mathcal{P})\right)  -f\left(
\mathrm{dist}(b_{e},\mathcal{P}^{\prime})\right)  \right]  \phi_{b_{e}}\\
&  \leq%
{\displaystyle\sum\limits_{b_{e}\in\mathcal{C}_{r}^{\mathcal{N}}}}
\left\Vert \frac{\partial f}{\partial x}\right\Vert _{\infty}\left\vert
\mathrm{dist}(b_{e},\mathcal{P})-\mathrm{dist}(b_{e},\mathcal{P}^{\prime
})\right\vert \phi_{b_{e}}\\
&  \leq\left(
{\displaystyle\sum\limits_{b_{e}\in\mathcal{C}_{r}^{\mathcal{N}}}}
\phi_{b_{e}}\right)  \left\Vert \frac{\partial f}{\partial x}\right\Vert
_{\infty}\left\Vert \mathcal{P}-\mathcal{P}^{\prime}\right\Vert \text{,}%
\end{align*}
where $\left\Vert \frac{\partial f}{\partial x}\right\Vert _{\infty}$ is the
$L_{\infty}$-norm of $\frac{\partial f}{\partial x}$.
\end{remark}

\begin{theorem}
\label{th:grad_bar}The multi-center function $\mathcal{H}$ is continuously
differentiable on $\left(  \mathbb{R}^{2}\right)  ^{m}\setminus\left(
\mathcal{D}_{\mathcal{C}_{r}^{\mathcal{N}}}\right)  ^{m}$, where
\[
\mathcal{D}_{\mathcal{C}_{r}^{\mathcal{N}}}\triangleq%
{\displaystyle\bigcup\nolimits_{b_{e}\in\mathcal{C}_{r}^{\mathcal{N}}}}
\left\{  q\in\mathbb{R}^{2}\mid\left\Vert b_{e}-q\right\Vert =R_{i},~\forall
i=1,\ldots,N\right\}
\]
is the discontinuity set of $f(\cdot)$ in $\mathbb{R}^{2}$. Moreover, for each
$h\in\left\{  1,\ldots,m\right\}  $%
\begin{equation}
\frac{\partial\mathcal{H}(\mathcal{P})}{\partial p_{h}}=%
{\displaystyle\sum\limits_{b_{e}\in V_{h}^{\mathcal{C}_{r}^{\mathcal{N}}%
}(\mathcal{P})}}
\frac{\partial}{\partial p_{h}}f\left(  \left\Vert b_{e}-p_{h}\right\Vert
\right)  \phi_{b_{e}}\text{.} \label{eq:def_gradcomp_bar}%
\end{equation}

\end{theorem}

\begin{proof}
The continuous differentiability of $\mathcal{H}$ on $\left(  \mathbb{R}%
^{2}\right)  ^{m}\setminus$ $\left(  \mathcal{D}_{\mathcal{C}_{r}%
^{\mathcal{N}}}\right)  ^{m}$ is a straight consequence of the same property
of $f(\cdot)$ on $\mathbb{R}^{2}\setminus$ $\mathcal{D}_{\mathcal{C}%
_{r}^{\mathcal{N}}}$. As concerns the gradient, using (\ref{eq:defH_bar_dist})
we have%
\begin{align*}
\frac{\partial\mathcal{H}(\mathcal{P})}{\partial p_{h}}  &  =\frac{\partial
}{\partial p_{h}}%
{\displaystyle\sum\limits_{b_{e}\in\mathcal{C}_{r}^{\mathcal{N}}}}
f\left(  \mathrm{dist}(b_{e},\mathcal{P})\right)  \phi_{b_{e}}\\
&  =%
{\displaystyle\sum\limits_{b_{e}\in\mathcal{C}_{r}^{\mathcal{N}}}}
\frac{\partial}{\partial p_{h}}f\left(  \mathrm{dist}(b_{e},\mathcal{P}%
)\right)  \phi_{b_{e}}\text{.}%
\end{align*}
In the hypothesis of Assumption~\ref{Assumption1}, we have that the index set
$I_{be}=\operatorname*{argmin}_{i\in\left\{  1,\ldots,m\right\}  }\left(
\left\Vert b_{e}-p_{i}\right\Vert \right)  $ has cardinality $1$. Hence,
$\forall b_{e}$ such that $I_{be}\equiv i\neq h$, $\frac{\partial}{\partial
p_{h}}f\left(  \mathrm{dist}(b_{e},\mathcal{P})\right)  =0$ and
\[
\frac{\partial\mathcal{H}(\mathcal{P})}{\partial p_{h}}=%
{\displaystyle\sum\limits_{\left\{  b_{e}\in\mathcal{C}_{r}^{\mathcal{N}}\mid
I_{be}\equiv h\right\}  }}
\frac{\partial}{\partial p_{h}}f\left(  \left\Vert b_{e}-p_{h}\right\Vert
\right)  \phi_{b_{e}}\text{,}%
\]
whereby the thesis follows using the definition of $V_{h}^{\mathcal{C}%
_{r}^{\mathcal{N}}}(\mathcal{P})$.
\end{proof}

The sensors' location-allocation problem can be addressed by means of a
gradient-like algorithm. If a continuous time implementation is looked for,
the following fictitious dynamics would be associated to the sensors'
positions%
\begin{equation}
\mathcal{\dot{P}}=\nabla\mathcal{H(P)}\text{.} \label{eq:graddyncont}%
\end{equation}
Unfortunately, this dynamics conveys some problems. It is well defined as long
as Assumption~\ref{Assumption1} and Theorem~\ref{th:grad_bar} are fulfilled,
but these hypotheses are, in fact, too stringent for the algorithm to work
properly. Indeed, they would require the evolution of the sensors to avoid any
position in the discontinuity set and the barycenters not to enter or exit the
Voronoi cells where they are at the initial time instant.

First of all we can reduce the analysis to continuously differentiable
functions to avoid issues related to the existence of the gradient. Still the
relaxation of Assumption~\ref{Assumption1} induces some problems on the
definition of the gradient. Barycenters on a boundary edge of a Voronoi cell
belong to all the cells sharing that edge. All sensors' positions producing
these configurations are discontinuity points for $\nabla\mathcal{H(P)}$.
Roughly speaking, the gradient takes different values depending on which cell
the shared barycenters are assumed to belong to. This fact makes the equation
(\ref{eq:graddyncont}) a set of differential equations with discontinuous
right-hand side.

The problem of shared barycenters can be solved by adding a lexicographic
criterion to the definitions based on the euclidean distance. Indeed, with
this criterion barycenters on boundary edges are allocated univocally to the
sensor having the lower index (w.r.t. the Lexicographic Order (L.O.)) among
sharing sensors. This fact allows us to define a genuine Voronoi partition, no
longer a covering, whose generic cell is given by (compare with
Definition~\ref{def:Voronoi_coll}):%
\begin{multline}
V_{i}^{\mathcal{C}_{r}^{\mathcal{N}}}(\mathcal{P})=\{b\in\mathcal{C}%
_{r}^{\mathcal{N}}\mid\left\Vert b-p_{i}\right\Vert \leq\left\Vert
b-p_{j}\right\Vert ~\forall p_{j}\in\mathcal{P}~\wedge\text{~}\\
\left\Vert b-p_{i}\right\Vert <\left\Vert b-p_{j}\right\Vert \text{ if
}j<i\text{ w.r.t. the L.O.}\mathcal{\}}\text{.} \label{eq:defVorcellLG}%
\end{multline}

We must now define a generalized (lexicographic) gradient of $\mathcal{H}$,
$\nabla_{l}\mathcal{H(P)}$, according to this new definition. $\forall
b_{e}\in V_{i}^{\mathcal{C}_{r}^{\mathcal{N}}}(\mathcal{P})\setminus\partial
V_{i}^{\mathcal{C}_{r}^{\mathcal{N}}}(\mathcal{P})$ we use the classical
formula given by (\ref{eq:def_gradcomp_bar}). $\forall b_{e}\in\partial
V_{i}^{\mathcal{C}_{r}^{\mathcal{N}}}(\mathcal{P})$ notice that the partial
derivative of $f$, $\frac{\partial}{\partial p_{h}}f\left(  \left\Vert
b_{e}-p_{h}\right\Vert \right)  $, exists and is well defined. In the light of
this remark we can write the $h$-th component of the generalized
(lexicographic) gradient of $\mathcal{H}$ as%
\begin{equation}
\frac{\partial_{l}\mathcal{H}(\mathcal{P})}{\partial p_{h}}\triangleq%
{\displaystyle\sum\limits_{b_{e}\in V_{h}^{\mathcal{C}_{r}^{\mathcal{N}}%
}(\mathcal{P})}}
\frac{\partial}{\partial p_{h}}f\left(  \left\Vert b_{e}-p_{h}\right\Vert
\right)  \phi_{b_{e}}\text{,} \label{eq:def_gradcomp_bar_LG}%
\end{equation}
which is formally equal to the formula (\ref{eq:def_gradcomp_bar}) provided by
Theorem~\ref{th:grad_bar}.

The differential equation using this new definition for the gradient, however,
does not imply the existence and uniqueness of the solution, and this proof
may turn out to be complex due to special sensors' and barycenters'
configurations. Moreover, the formula (\ref{eq:def_gradcomp_bar_LG}) accounts
only for the infinitesimal perturbations of sensors' position not inducing any
barycenters to enter or exit Voronoi cells, hence changing their allocation.

In order to simplify the convergence proof and to provide an algorithm which
is more suitable for a realistic implementation, we consider here a
discrete-time version of the gradient algorithm. In our case, the
discrete-time implementation can overcome the problems with discontinuity
gradients due to the properties of the function $\mathcal{H}$ and its
discontinuity points, as it is shown by the following theorem.

\begin{theorem}
\label{Th:discrevol_colR2}Consider the following discrete-time evolution for
the sensors' positions%
\begin{equation}
\mathcal{P}^{(k+1)}=\mathcal{P}^{(k)}+\delta_{k}\nabla_{l}\mathcal{H}%
(\mathcal{P}^{(k)})\text{,} \label{eq:graddyndisc}%
\end{equation}
where the $h$-th component of $\nabla_{l}\mathcal{H}$ is given by
(\ref{eq:def_gradcomp_bar_LG}) and $\mathcal{H}:\mathbb{R}^{2m}\rightarrow
\mathbb{R}$ as in (\ref{eq:defH_bar_Vor}). If $f(\cdot)$ has locally bounded
second derivatives, then, for suitable $\delta_{k}$, $\mathcal{P}^{(k)}$ lies
in a bounded set and

\begin{description}
\item[i)] $\mathcal{H}(\mathcal{P}^{(k)})$ is monotonically nondecreasing;

\item[ii)] $\mathcal{P}^{(k)}$ converges to the set of critical points of
$\mathcal{H}$.
\end{description}
\end{theorem}

\begin{proof}
It is easy to see that there exists a ball $B\supseteq\mathcal{N}$ such that,
if $p_{i}\in\partial B$, then $\frac{\partial_{l}\mathcal{H}(\mathcal{P}%
)}{\partial p_{i}}$ points inside $B$. Thus the fact that $\mathcal{P}^{(k)}%
$\ is bounded for $\delta_{k}$ sufficiently small, easily follows. According
to the definition in equation (\ref{eq:defVorcellLG})\ of a Voronoi cell using
the lexicographic rule, we can define
\[
\mathcal{H}_{p_{i}}(\mathcal{P})=%
{\displaystyle\sum\limits_{b_{e}\in V_{i}^{\mathcal{C}_{r}^{\mathcal{N}}%
}(\mathcal{P})}}
f\left(  \left\Vert b_{e}-p_{i}\right\Vert \right)  \phi_{b_{e}}\text{,}%
\]
hence we can write $\mathcal{H}(\mathcal{P})=%
{\displaystyle\sum\limits_{i=1}^{m}}
\mathcal{H}_{p_{i}}(\mathcal{P})$. We must prove that $\mathcal{H}_{p_{i}%
}(\mathcal{P}^{(k+1)})\geq\mathcal{H}_{p_{i}}(\mathcal{P}^{(k)})$ for $\forall
p_{i}\in\mathcal{P}^{(k)}$ obeying the discrete-time evolution
(\ref{eq:graddyndisc}) and that $\exists p_{i}\in\mathcal{P}^{(k)}$ such that
$\mathcal{H}_{p_{i}}(\mathcal{P}^{(k+1)})>\mathcal{H}_{p_{i}}(\mathcal{P}%
^{(k)})$ if $\mathcal{P}^{(k)}$ does not belong to the set of critical points
of $\mathcal{H}$. Define $\mathcal{H}_{p_{i}}^{u}(\mathcal{P}^{(k+1)})=%
{\displaystyle\sum\limits_{b_{e}\in V_{i}^{\mathcal{C}_{r}^{\mathcal{N}}%
}(\mathcal{P}^{(k)})}}
f\left(  \left\Vert b_{e}-p_{i}^{(k+1)}\right\Vert \right)  \phi_{b_{e}}$,
that is the cost $\mathcal{H}_{p_{i}}(\mathcal{\cdot})$ computed by using the
new sensors' positions $\mathcal{P}^{(k+1)}$ but the old allocation of
barycenters to sensors, i.e. the Voronoi partition generated by $\mathcal{P}%
^{(k)}$. Therefore, we write%
\begin{align}
\mathcal{H}_{p_{i}}^{u}(\mathcal{P}^{(k+1)})  &  =\mathcal{H}_{p_{i}%
}(\mathcal{P}^{(k)}+\delta_{k}^{i}\nabla_{l}\mathcal{H}(\mathcal{P}%
^{(k)}))\nonumber\\
&  =\mathcal{H}_{p_{i}}(\mathcal{P}^{(k)})+\delta_{k}^{i}\left(
\frac{\partial\mathcal{H}_{p_{i}}(\mathcal{P}^{(k)})}{\partial p_{i}}%
\cdot\left(  \nabla_{l}\mathcal{H}(\mathcal{P}^{(k)})\right)  _{i}\right)
+o\left(  (\delta_{k}^{i})^{2}\right) \nonumber\\
&  =\mathcal{H}_{p_{i}}(\mathcal{P}^{(k)})+\delta_{k}^{i}\left\Vert
\frac{\partial\mathcal{H}_{p_{i}}(\mathcal{P}^{(k)})}{\partial p_{i}%
}\right\Vert ^{2}+o\left(  (\delta_{k}^{i})^{2}\right)  \text{,}
\label{eq:Hu_colR2}%
\end{align}
where the (generalized) gradient is computed by means of the lexicographic
assignment based on the Voronoi partition generated by $\mathcal{P}^{(k)}$. If
the Voronoi partitions generated by $\mathcal{P}^{(k+1)}$ and $\mathcal{P}%
^{(k)}$\ coincide (same barycenters' allocation), then $\mathcal{H}_{p_{i}%
}(\mathcal{P}^{(k+1)})=\mathcal{H}_{p_{i}}^{u}(\mathcal{P}^{(k+1)})$ and we
can assert directly that $\mathcal{H}_{p_{i}}(\mathcal{P}^{(k+1)}%
)\geq\mathcal{H}_{p_{i}}(\mathcal{P}^{(k)})$ or $\mathcal{H}_{p_{i}%
}(\mathcal{P}^{(k+1)})>\mathcal{H}_{p_{i}}(\mathcal{P}^{(k)})$ if
$\frac{\partial\mathcal{H}_{p_{i}}(\mathcal{P}^{(k)})}{\partial p_{i}}\neq0$,
that is if $\mathcal{P}^{(k)}$ does not belong to the set of critical points
of $\mathcal{H}$. It is worth noting that similar remarks about the strict
inequality apply also in what follows and they will not be repeated. If the
barycenters' allocation is different, $\mathcal{H}_{p_{i}}(\mathcal{P}%
^{(k+1)})$ may be smaller than $\mathcal{H}_{p_{i}}^{u}(\mathcal{P}^{(k+1)})$
and we cannot say anything about its relation with $\mathcal{H}_{p_{i}%
}(\mathcal{P}^{(k)})$. This fact is due to the presence of barycenters
changing allocation during the sensors' evolution, hence the cells involved in
the change cannot be considered independently. Let us consider for simplicity
only one barycenter $b_{e}\in V_{j}^{\mathcal{C}_{r}^{\mathcal{N}}%
}(\mathcal{P}^{(k)})$ and suppose that at step $k+1$ $b_{e}\in V_{i}%
^{\mathcal{C}_{r}^{\mathcal{N}}}(\mathcal{P}^{(k+1)})$. With this assumption
no barycenters can enter or exit the union of the two cells but $b_{e}$. We
have%
\begin{align}
\mathcal{H}_{p_{i}}(\mathcal{P}^{(k+1)})  &  =\mathcal{H}_{p_{i}}%
^{u}(\mathcal{P}^{(k+1)})+f\left(  \left\Vert b_{e}-p_{i}^{(k+1)}\right\Vert
\right)  \phi_{b_{e}}\nonumber\\
\mathcal{H}_{p_{j}}(\mathcal{P}^{(k+1)})  &  =\mathcal{H}_{p_{j}}%
^{u}(\mathcal{P}^{(k+1)})-f\left(  \left\Vert b_{e}-p_{j}^{(k+1)}\right\Vert
\right)  \phi_{b_{e}}\text{.} \label{eq:Hu_contrib}%
\end{align}
In other words, $\mathcal{H}_{p_{i}}(\mathcal{P}^{(k+1)})$ is grown w.r.t. the
ideal value $\mathcal{H}_{p_{i}}^{u}(\mathcal{P}^{(k+1)})$ due to the
allocation of $b_{e}$ to $p_{i}^{(k+1)}$, whereas, $\mathcal{H}_{p_{j}%
}(\mathcal{P}^{(k+1)})$ is decreased w.r.t. $\mathcal{H}_{p_{i}}%
^{u}(\mathcal{P}^{(k+1)})$ by the contribution that $b_{e}$ would have given
if it were allocated to $p_{j}^{(k+1)}$ even at the step $k+1$. It is worth
noting that if $b_{e}$ changes allocation from $p_{j}^{(k+1)}$ to
$p_{i}^{(k+1)}$, then $\left\Vert b_{e}-p_{i}^{(k+1)}\right\Vert
\leq\left\Vert b_{e}-p_{j}^{(k+1)}\right\Vert $ (the equality holds only if
$b_{e}\in\partial V_{i}^{\mathcal{C}_{r}^{\mathcal{N}}}(\mathcal{P}%
^{(k+1)})\cap\partial V_{j}^{\mathcal{C}_{r}^{\mathcal{N}}}(\mathcal{P}%
^{(k+1)})$ and, being $i<j$, the allocation is induced by the lexicographic
rule). Therefore, due to the monotonicity of $f$, $f\left(  \left\Vert
b_{e}-p_{i}^{(k+1)}\right\Vert \right)  \geq f\left(  \left\Vert b_{e}%
-p_{j}^{(k+1)}\right\Vert \right)  $. Summing up, we have%
\begin{align*}
\mathcal{H}_{p_{i}}(\mathcal{P}^{(k+1)})+\mathcal{H}_{p_{j}}(\mathcal{P}%
^{(k+1)})  &  \geq\mathcal{H}_{p_{i}}^{u}(\mathcal{P}^{(k+1)})+\mathcal{H}%
_{p_{j}}^{u}(\mathcal{P}^{(k+1)})\\
&  =\mathcal{H}_{p_{i}}(\mathcal{P}^{(k)})+\delta_{k}^{i}\left\Vert
\frac{\partial\mathcal{H}_{p_{i}}(\mathcal{P}^{(k)})}{\partial p_{i}%
}\right\Vert ^{2}+o\left(  (\delta_{k}^{i})^{2}\right) \\
&  +\mathcal{H}_{p_{j}}(\mathcal{P}^{(k)})+\delta_{k}^{j}\left\Vert
\frac{\partial\mathcal{H}_{p_{j}}(\mathcal{P}^{(k)})}{\partial p_{j}%
}\right\Vert ^{2}+o\left(  (\delta_{k}^{j})^{2}\right)  \text{.}%
\end{align*}
There exist some constants $c_{k}^{ij},\delta_{k}^{ij}>0$ such that%
\[
\mathcal{H}_{p_{i}}(\mathcal{P}^{(k+1)})+\mathcal{H}_{p_{j}}(\mathcal{P}%
^{(k+1)})\geq\mathcal{H}_{p_{i}}(\mathcal{P}^{(k)})+\mathcal{H}_{p_{j}%
}(\mathcal{P}^{(k)})+\delta_{k}^{ij}c_{k}^{ij}+o\left(  (\delta_{k}^{ij}%
)^{2}\right)  \text{.}%
\]
Extending the same reasoning to more complex configurations involving more
than one common boundary edge and more than two neighboring cells, we can say
that $\exists c_{k},\delta_{k}>0$ such that%
\[
\mathcal{H}(\mathcal{P}^{(k+1)})=%
{\displaystyle\sum\limits_{i=1}^{m}}
\mathcal{H}_{p_{i}}(\mathcal{P}^{(k+1)})\geq\mathcal{H}(\mathcal{P}%
^{(k)})+\delta_{k}c_{k}+o(\delta_{k}^{2})\text{.}%
\]
This proves assertion i). Assertion ii) can be proved by exploiting the
results of Section 3.2 of \cite{CortesESAIM05}. More precisely, using the fact
that $\mathcal{P}^{(k)}$ is bounded and $f$ has locally bounded second
derivatives, then there exists $\bar{\delta}>0$ such that we can choose
$\delta_{k}\geq\bar{\delta}$. Then we conclude, using Proposition 3.4 of
\cite{CortesESAIM05}, that assertion ii) is true.
\end{proof}

\begin{remark}
[Distributed Implementation]The use of a gradient ascent algorithm based on a
Voronoi partition, allows us to solve not only a static deployment problem,
but also a dynamic one. As shown in \cite{CortesESAIM05}, this kind of
algorithms is spatially distributed, with the meaning that the $i$-th sensor
needs only to know the position of its neighbors in order to determine the
boundary of its cell and, hence, to compute $\frac{\partial\mathcal{H}_{p_{i}%
}}{\partial p_{i}}$. For the same reason the $i$-th sensor can choose the
value of the step-size $\delta_{k}^{i}$ independently of the other sensors
simply performing locally a classical line search algorithm.\ This property
makes the algorithm suitable for a spatially distributed implementation.

The independence in the choice of the step-sizes $\delta_{k}^{i}$ is obviously
preserved in each period $k$, as long as a synchronous implementation is
considered. In this case sensors have access to a global clock, or perform a
synchronization algorithm. At the beginning of the the $k$-th period (instant
$t_{k}$), all sensors are idle, build their Voronoi cells and compute their
gradients and step-sizes, then they move until, at most, the end of the period
(instant $t_{k+1}$). If, instead, an asynchronous implementation is
considered, further hypotheses are necessary to ensure that independence is
preserved. Unfortunately, the discontinuity of the gradient prevent us from
using the results of \cite{TsitsiklisTAC86}. But, if a sensor has the
capability to detect when its neighbors start and stop moving and when a new
sensor joins the neighborhood, the asynchronous algorithm presented in
\cite{CortesTRA04} (Table IV), can be applied, thus automatically recovering
the independence.
\end{remark}

\begin{remark}
In the previous theorem, for sake of simplicity, we did not consider
degenerate configurations where different sensors have the same position
($p_{i}=p_{j}$ for $i\neq j$). But it can be proved that if the initial
positions of sensors are not degenerate, sensors can always choose a suitable
$\delta_{k}^{i}$ to avoid the occurrence of these configurations.
\end{remark}

\subsection{Sensors Moving on the Network\label{sec:sensonN}}

This section is devoted to the case of sensors constrained to move on the
network and sensing a collapsed network. Therefore, these results are suitable
for an implementation of the second step of our procedure on hardware with
limited computational capabilities.

We still assume $f(\cdot)$ to be a continuously differentiable function and we
make use of the lexicographic criterion for the barycenter allocation. As
concerns sensors' motion, however, we cannot use directly the gradient since
the sensors have to remain on the network. We must consider now the
directional derivative of $\mathcal{H}$\ along the edges of the network.

Following the guidelines of the previous section, the following theorem can be proved.

\begin{theorem}
Given a network $\mathcal{N}=(\mathcal{V},\mathcal{S})$ and the related
$r$-collapsed network $\mathcal{C}_{r}^{\mathcal{N}}$, the multi-center
function $\mathcal{H}$ is continuously differentiable almost everywhere on
$\mathcal{N}^{m}$. In particular, on each open segment $s_{ij}^{o}$ such that
$s_{ij}\in\mathcal{S}$, given the unit vector $w_{ij}$ such that $s_{ij}\cdot
w_{ij}=\left\Vert s_{ij}\right\Vert $, the directional derivative in $p_{h}\in
s_{ij}^{o}$ along $w_{ij}$ is
\begin{align}
D_{w_{ij}}\mathcal{H(P)}[p_{h}]  &  =\left(  \frac{\partial_{l}\mathcal{H}%
}{\partial p_{h}}(\mathcal{P})\cdot w_{ij}\right)  w_{ij}%
\label{eq:def_direcder_on_s}\\
\frac{\partial_{l}\mathcal{H}}{\partial p_{h}}(\mathcal{P})  &  =%
{\displaystyle\sum\limits_{b_{e}\in V_{h}^{\mathcal{C}_{r}^{\mathcal{N}}%
}(\mathcal{P})}}
\frac{\partial}{\partial p_{h}}f\left(  \left\Vert b_{e}-p_{h}\right\Vert
\right)  \phi_{b_{e}}\text{.}\nonumber
\end{align}

\end{theorem}

The directional derivative is a multivalued function on the vertices of the
network as more than one edge can share the same vertex, but we need a
univocal definition. Hence, we fix a choice rule such that the directional
derivative in a vertex is given by the maximum among all the derivatives
defined for each possible direction that does not lead the sensor out of the
network. If all the directional derivatives in a vertex point outward the
network, then the derivative is set equal to zero.

\begin{definition}
\label{def:direcder_complete}Given the set
\begin{multline}
S_{v_{i}}=\left\{  s\in\mathcal{S}\mid\exists v_{j}\in\mathcal{V},\exists
\bar{\delta}>0\text{ s.t. }s=[v_{i},v_{j}]\vee s=[v_{j},v_{i}],\right. \\
\left.  \forall\delta\in\lbrack0,\bar{\delta}]~v_{i}+\delta D_{w_{ij}%
}\mathcal{H(P)}[v_{i}]\in\mathcal{N}\right\}  \text{,} \label{eq:def_setSvi}%
\end{multline}
we define the directional derivative of $\mathcal{H(P)}$ in any point
$p_{h}\in\mathcal{N}$ as follows%
\begin{equation}
\widetilde{D}_{h}\mathcal{H(P)=}\left\{
\begin{array}
[c]{cc}%
\begin{array}
[c]{c}%
D_{w_{ij}}\mathcal{H(P)}[p_{h}]\text{ given by (\ref{eq:def_direcder_on_s})}\\
\text{ }%
\end{array}
&
\begin{array}
[c]{c}%
\forall p_{h}\in\mathcal{N}\setminus\mathcal{V}\\
\text{ }%
\end{array}
\\%
\begin{array}
[c]{c}%
\max\limits_{s_{ij}\in S_{v_{i}}}D_{w_{ij}}\mathcal{H(P)}[p_{h}]\\
\text{ }%
\end{array}
&
\begin{array}
[c]{c}%
\begin{array}
[c]{c}%
\forall p_{h}\equiv v_{i}\in\mathcal{V}\\
S_{v_{i}}\neq\emptyset
\end{array}
\\
\text{ }%
\end{array}
\\
0 &
\begin{array}
[c]{c}%
\forall p_{h}\equiv v_{i}\in\mathcal{V}\\
S_{v_{i}}=\emptyset
\end{array}
\end{array}
\right.  \label{eq:def_direcder_complete}%
\end{equation}

\end{definition}

We can now define the discrete-time gradient-like algorithm.

\begin{theorem}
\label{Th:discrevol_colN}Consider the following discrete-time evolution for
the sensors' positions%
\begin{equation}
\mathcal{P}^{(k+1)}=\mathcal{P}^{(k)}+\delta_{k}\widetilde{D}\mathcal{H}%
(\mathcal{P}^{(k)})\text{,}%
\end{equation}
where the $h$-th component of $\widetilde{D}\mathcal{H}$ is given by
(\ref{eq:def_direcder_complete}) and $\mathcal{H}:\mathcal{N}^{m}%
\rightarrow\mathbb{R}$ as in (\ref{eq:defH_bar_Vor}). If $f(\cdot)$ has
locally bounded second derivatives, then, for suitable $\delta_{k}$,
$\mathcal{P}^{(k)}$ lies in a bounded set and

\begin{description}
\item[i)] $\mathcal{H}(\mathcal{P}^{(k)})$ is monotonically nondecreasing;

\item[ii)] $\mathcal{P}^{(k)}$ converges to the set of critical points of
$\mathcal{H}$.
\end{description}
\end{theorem}

\begin{proof}
The proof is essentially the same of the Theorem \ref{Th:discrevol_colR2},
except for the use of the derivative (\ref{eq:def_direcder_complete}) in place
of the gradient (\ref{eq:def_gradcomp_bar_LG}). In particular equality
(\ref{eq:Hu_colR2}) becomes%
\begin{align*}
\mathcal{H}_{p_{i}}^{u}(\mathcal{P}^{(k+1)})  &  =\mathcal{H}_{p_{i}%
}(\mathcal{P}^{(k)}+\delta_{k}^{i}\widetilde{D}\mathcal{H}(\mathcal{P}%
^{(k)}))\\
&  =\mathcal{H}_{p_{i}}(\mathcal{P}^{(k)})+\delta_{k}^{i}\left(
\frac{\partial\mathcal{H}_{p_{i}}(\mathcal{P}^{(k)})}{\partial p_{i}}%
\cdot\left(  \widetilde{D}\mathcal{H}(\mathcal{P}^{(k)})\right)  _{i}\right)
+o\left(  (\delta_{k}^{i})^{2}\right) \\
&  =\mathcal{H}_{p_{i}}(\mathcal{P}^{(k)})+\delta_{k}^{i}\left(
\frac{\partial\mathcal{H}_{p_{i}}(\mathcal{P}^{(k)})}{\partial p_{i}}\cdot
w_{hl}\right)  ^{2}+o\left(  (\delta_{k}^{i})^{2}\right)  \text{,}%
\end{align*}
for $p_{i}$ moving along $w_{hl}$.
\end{proof}

\section{Deployment over a Full Network\label{sec:FullNet}}

In this section we consider a more accurate version of the second step of the
optimization procedure, namely sensors constrained on the network and sensing
the full network. To start with, let us define the boundary of a Voronoi cell
as $\partial V_{i}^{\mathcal{N}}(\mathcal{P})=\{q\in\mathcal{N}\mid\left\Vert
q-p_{i}\right\Vert =\left\Vert q-p_{j}\right\Vert ,$~$\exists p_{j}%
\in\mathcal{P\}}$, and the instantaneous discontinuity set of $f(\cdot)$ as%
\[
\mathcal{D}_{\mathcal{N}}(\mathcal{P})\triangleq%
{\displaystyle\bigcup\nolimits_{p_{j}\in\mathcal{P}}}
\left\{  q\in\mathcal{N}\mid\left\Vert q-p_{j}\right\Vert =R_{i},~\forall
i=1,\ldots,N\right\}  \text{.}%
\]

\begin{assumption}
\label{Ass_fullnet}We make the following assumptions:

\begin{description}
\item[i)] orthogonality assumption: $\forall h,k\in\left\{  1,\ldots
,n\right\}  $, $\forall i\in\left\{  1,\ldots,m\right\}  $ and for any segment
$s=[a,b]\subseteq s_{hk}\in\mathcal{S}$ with $a\neq b$, $s\not \subset
\partial V_{i}^{\mathcal{N}}(\mathcal{P})$;

\item[ii)] $\partial V_{i}^{\mathcal{N}}(\mathcal{P})\cap\mathcal{D}%
_{\mathcal{N}}(\mathcal{P})=\emptyset$, $\forall i\in\left\{  1,\ldots
,m\right\}  $;

\item[iii)] $\mathcal{V}\cap\mathcal{D}_{\mathcal{N}}(\mathcal{P})=\emptyset$;

\item[iv)] $\forall h,k\in\left\{  1,\ldots,n\right\}  $, $\forall
i\in\left\{  1,\ldots,m\right\}  $, $\forall q\in s_{hk}\cap V_{i}%
^{\mathcal{N}}(\mathcal{P})$, if $(q-p_{i})\cdot(v_{h}-v_{k})=0\Rightarrow
\left\Vert q-p_{i}\right\Vert \notin\left\{  R_{1},\ldots,R_{N}\right\}  $;
\end{description}
\end{assumption}

With the orthogonality assumption the expression (\ref{eq:defH_Vor})
simplifies to%
\begin{equation}
\mathcal{H}(\mathcal{P})=%
{\displaystyle\sum\limits_{i=1}^{m}}
{\displaystyle\int\nolimits_{V_{i}^{\mathcal{N}}(\mathcal{P})}}
f\left(  \left\Vert q-p_{i}\right\Vert \right)  \phi\left(  q\right)
dq\text{.} \label{eq:defH_Vor_red}%
\end{equation}

\begin{theorem}
Given a network $\mathcal{N}=(\mathcal{V},\mathcal{S})$ if Assumption
\ref{Ass_fullnet} holds, the multi-center function $\mathcal{H}$ is
continuously differentiable almost everywhere on $\mathcal{N}^{m}$. In
particular, on each open segment $s_{ij}^{o}$ such that $s_{ij}\in\mathcal{S}%
$, given the unit vector $w_{ij}$ such that $s_{ij}\cdot w_{ij}=\left\Vert
s_{ij}\right\Vert $, the directional derivative in $p_{h}\in s_{ij}^{o}$ along
$w_{ij}$ is
\begin{gather}
D_{w_{ij}}\mathcal{H(P)}[p_{h}]=\left(  \frac{\partial\mathcal{H}}{\partial
p_{h}}(\mathcal{P})\cdot w_{ij}\right)  w_{ij} \label{eq:def_direcderNET_on_s}%
\\
\frac{\partial\mathcal{H}}{\partial p_{h}}(\mathcal{P})=%
{\displaystyle\sum\limits_{k=1}^{M_{h}(\mathcal{P})}}
\mathcal{I}_{k}\nonumber\\
\mathcal{I}_{k}=\left\Vert b_{k}-a_{k}\right\Vert
{\displaystyle\int\nolimits_{[0,1]\backslash\left\{  t_{1,1}^{k},t_{1,2}%
^{k},\ldots,t_{N,1}^{k},t_{N,2}^{k}\right\}  }}
\frac{\partial}{\partial\nu}f\left(  \nu\right)  \frac{p_{h}-\gamma_{k}%
(t)}{\left\Vert p_{h}-\gamma_{k}(t)\right\Vert }\phi\left(  \gamma
_{k}(t)\right)  dt\nonumber\\
+\left\Vert b_{k}-a_{k}\right\Vert
{\displaystyle\sum\limits_{\alpha=1}^{N}}
\left(  f_{\alpha+1}(R_{\alpha})-f_{\alpha}(R_{\alpha})\right)
{\displaystyle\sum\limits_{j=1}^{2}}
\frac{p_{h}-\gamma_{k}(t_{\alpha,j}^{k})}{\left\Vert p_{h}-\gamma
_{k}(t_{\alpha,j}^{k})\right\Vert }\phi\left(  \gamma_{k}(t_{\alpha,j}%
^{k})\right)  \text{,}\nonumber
\end{gather}
where $\gamma_{k}(t)=a_{k}+\left(  b_{k}-a_{k}\right)  t,$ $t\in\lbrack0,1]$
is a parameterization for the $k$-th segment $[a_{k},b_{k}]\in V_{h}%
^{\mathcal{N}}(\mathcal{P})$, $M_{h}(\mathcal{P})$ is the number of segments
in $V_{h}^{\mathcal{N}}(\mathcal{P})$ and $t_{\alpha,j}^{k}\in\lbrack
0,1],~j\in\{1,2\}$ are the zeros of $\left\Vert \gamma_{k}(t)-p_{h}\right\Vert
-R_{\alpha}=0$ (if any).
\end{theorem}

\begin{proof}
Consider the gradient of $\mathcal{H}(\mathcal{P})$ in the form
(\ref{eq:defH_Vor_red})%
\begin{align}
\frac{\partial\mathcal{H}}{\partial p_{h}}(\mathcal{P})  &  =\frac{\partial
}{\partial p_{h}}%
{\displaystyle\sum\limits_{i=1}^{m}}
{\displaystyle\int\nolimits_{V_{i}^{\mathcal{N}}(\mathcal{P})}}
f\left(  \left\Vert q-p_{i}\right\Vert \right)  \phi\left(  q\right)
dq\nonumber\\
&  =\frac{\partial}{\partial p_{h}}%
{\displaystyle\int\nolimits_{V_{h}^{\mathcal{N}}(\mathcal{P})}}
f\left(  \left\Vert q-p_{h}\right\Vert \right)  \phi\left(  q\right)
dq\nonumber\\
&  +\frac{\partial}{\partial p_{h}}%
{\displaystyle\sum\limits_{\substack{i=1\\i\neq h}}^{m}}
{\displaystyle\int\nolimits_{V_{i}^{\mathcal{N}}(\mathcal{P})}}
f\left(  \left\Vert q-p_{i}\right\Vert \right)  \phi\left(  q\right)
dq\text{.} \label{eq:grad_full}%
\end{align}
Let us consider the second term of (\ref{eq:grad_full}) for each $i\neq h$%
\begin{align}
&  \frac{\partial}{\partial p_{h}}%
{\displaystyle\int\nolimits_{V_{i}^{\mathcal{N}}(\mathcal{P})}}
f\left(  \left\Vert q-p_{i}\right\Vert \right)  \phi\left(  q\right)
dq\nonumber\\
&  =\lim_{\varepsilon\rightarrow0}\frac{1}{\left\Vert \varepsilon\right\Vert
}\left(
{\displaystyle\int\nolimits_{V_{i}^{\mathcal{N}}(\left\{  p_{1},\ldots
,p_{h}+\varepsilon,\ldots,p_{m}\right\}  )}}
f\left(  \left\Vert q-p_{i}\right\Vert \right)  \phi\left(  q\right)  dq-%
{\displaystyle\int\nolimits_{V_{i}^{\mathcal{N}}(\mathcal{P})}}
f\left(  \left\Vert q-p_{i}\right\Vert \right)  \phi\left(  q\right)
dq\right) \nonumber\\
&  =\lim_{\varepsilon\rightarrow0}\frac{1}{\left\Vert \varepsilon\right\Vert
}\left(
{\displaystyle\int\nolimits_{\Delta V_{ih}^{\mathcal{N}}(\mathcal{P})^{+}}}
f\left(  \left\Vert q-p_{i}\right\Vert \right)  \phi\left(  q\right)  dq-%
{\displaystyle\int\nolimits_{\Delta V_{ih}^{\mathcal{N}}(\mathcal{P})^{-}}}
f\left(  \left\Vert q-p_{i}\right\Vert \right)  \phi\left(  q\right)
dq\right)  \text{,} \label{eq:secterm_full}%
\end{align}
where $\Delta V_{ih}^{\mathcal{N}}(\mathcal{P})^{+}=V_{i}^{\mathcal{N}%
}(\left\{  p_{1},\ldots,p_{h}+\varepsilon,\ldots,p_{m}\right\}  )\setminus
V_{i}^{\mathcal{N}}(\mathcal{P})$ and\linebreak\ $\Delta V_{ih}^{\mathcal{N}%
}(\mathcal{P})^{-}=V_{i}^{\mathcal{N}}(\mathcal{P})\setminus V_{i}%
^{\mathcal{N}}(\{p_{1},\ldots,p_{h}+\varepsilon,\ldots,p_{m}\})$. It is worth
noting that $V_{i}^{\mathcal{N}}(\{p_{1},\ldots,p_{h}+\varepsilon,\ldots
,p_{m}\})$ can be different from $V_{i}^{\mathcal{N}}(\mathcal{P})$ if and
only if $p_{i}\in\mathcal{N}_{\mathcal{G}_{D}}(p_{h},\mathcal{P}%
)$.\footnote{With $\mathcal{N}_{\mathcal{G}_{D}}(p_{h},\mathcal{P})$ we
represent the set of neighbors of $p_{h}$ in $\mathcal{P}$. The neighboring
property is given w.r.t. the proximity graph $\mathcal{G}_{D}(\mathcal{P})$,
that is the Delaunay graph associated to the Voronoi partition induced by
$\mathcal{P}$.} Let us consider now the first term of (\ref{eq:grad_full})%
\begin{align}
&  \frac{\partial}{\partial p_{h}}%
{\displaystyle\int\nolimits_{V_{h}^{\mathcal{N}}(\mathcal{P})}}
f\left(  \left\Vert q-p_{h}\right\Vert \right)  \phi\left(  q\right)
dq\nonumber\\
&  =\lim_{\varepsilon\rightarrow0}\frac{1}{\left\Vert \varepsilon\right\Vert
}\left(
{\displaystyle\int\nolimits_{V_{h}^{\mathcal{N}}(\left\{  p_{1},\ldots
,p_{h}+\varepsilon,\ldots,p_{m}\right\}  )}}
f\left(  \left\Vert q-\left(  p_{h}+\varepsilon\right)  \right\Vert \right)
\phi\left(  q\right)  dq-%
{\displaystyle\int\nolimits_{V_{h}^{\mathcal{N}}(\mathcal{P})}}
f\left(  \left\Vert q-p_{h}\right\Vert \right)  \phi\left(  q\right)
dq\right) \nonumber\\
&  =\lim_{\varepsilon\rightarrow0}\frac{1}{\left\Vert \varepsilon\right\Vert }%
{\displaystyle\int\nolimits_{V_{h}^{\mathcal{N}}(\left\{  p_{1},\ldots
,p_{h}+\varepsilon,\ldots,p_{m}\right\}  )}}
\left(  f\left(  \left\Vert q-\left(  p_{h}+\varepsilon\right)  \right\Vert
\right)  -f\left(  \left\Vert q-p_{h}\right\Vert \right)  \right)  \phi\left(
q\right)  dq\nonumber\\
&  +\lim_{\varepsilon\rightarrow0}\frac{1}{\left\Vert \varepsilon\right\Vert
}\left(
{\displaystyle\int\nolimits_{V_{h}^{\mathcal{N}}(\left\{  p_{1},\ldots
,p_{h}+\varepsilon,\ldots,p_{m}\right\}  )}}
f\left(  \left\Vert q-p_{h}\right\Vert \right)  \phi\left(  q\right)  dq-%
{\displaystyle\int\nolimits_{V_{h}^{\mathcal{N}}(\mathcal{P})}}
f\left(  \left\Vert q-p_{h}\right\Vert \right)  \phi\left(  q\right)
dq\right) \nonumber
\end{align}%
\begin{align}
&  =\lim_{\varepsilon\rightarrow0}\frac{1}{\left\Vert \varepsilon\right\Vert }%
{\displaystyle\int\nolimits_{V_{h}^{\mathcal{N}}(\mathcal{P})}}
\left(  f\left(  \left\Vert q-\left(  p_{h}+\varepsilon\right)  \right\Vert
\right)  -f\left(  \left\Vert q-p_{h}\right\Vert \right)  \right)  \phi\left(
q\right)  dq\nonumber\\
&  +\lim_{\varepsilon\rightarrow0}\frac{1}{\left\Vert \varepsilon\right\Vert }%
{\displaystyle\int\nolimits_{\Delta V_{hh}^{\mathcal{N}}(\mathcal{P})^{+}}}
\left(  f\left(  \left\Vert q-\left(  p_{h}+\varepsilon\right)  \right\Vert
\right)  -f\left(  \left\Vert q-p_{h}\right\Vert \right)  \right)  \phi\left(
q\right)  dq\nonumber
\end{align}%
\begin{align}
&  -\lim_{\varepsilon\rightarrow0}\frac{1}{\left\Vert \varepsilon\right\Vert }%
{\displaystyle\int\nolimits_{\Delta V_{hh}^{\mathcal{N}}(\mathcal{P})^{-}}}
\left(  f\left(  \left\Vert q-\left(  p_{h}+\varepsilon\right)  \right\Vert
\right)  -f\left(  \left\Vert q-p_{h}\right\Vert \right)  \right)  \phi\left(
q\right)  dq\nonumber\\
&  +\lim_{\varepsilon\rightarrow0}\frac{1}{\left\Vert \varepsilon\right\Vert
}\left(
{\displaystyle\int\nolimits_{\Delta V_{hh}^{\mathcal{N}}(\mathcal{P})^{+}}}
f\left(  \left\Vert q-p_{h}\right\Vert \right)  \phi\left(  q\right)  dq-%
{\displaystyle\int\nolimits_{\Delta V_{hh}^{\mathcal{N}}(\mathcal{P})^{-}}}
f\left(  \left\Vert q-p_{h}\right\Vert \right)  \phi\left(  q\right)
dq\right)  \text{,} \label{eq:firstterm_full}%
\end{align}
where $\Delta V_{hh}^{\mathcal{N}}(\mathcal{P})^{+}=V_{h}^{\mathcal{N}%
}(\left\{  p_{1},\ldots,p_{h}+\varepsilon,\ldots,p_{m}\right\}  )\setminus
V_{h}^{\mathcal{N}}(\mathcal{P})$ and \linebreak$\Delta V_{hh}^{\mathcal{N}%
}(\mathcal{P})^{-}=V_{h}^{\mathcal{N}}(\mathcal{P})\setminus V_{h}%
^{\mathcal{N}}(\left\{  p_{1},\ldots,p_{h}+\varepsilon,\ldots,p_{m}\right\}
)$. Now we want to prove that the sum of the second term of
(\ref{eq:grad_full}) and the last term of (\ref{eq:firstterm_full}) is null.
First of all, recall that the sum in the second term of (\ref{eq:grad_full})
can be limited to the cells in the neighborhood of the $h$-th cell, namely
$\forall i\in I_{h}$ with $I_{h}=\left\{  j\in\{1,\ldots,m\}\mid p_{j}%
\in\mathcal{N}_{\mathcal{G}_{D}}(p_{h},\mathcal{P})\right\}  $. This fact
implies that $%
{\displaystyle\bigcup\nolimits_{i\in I_{h}}}
\left(  \Delta V_{ih}^{\mathcal{N}}(\mathcal{P})^{+}\cup\Delta V_{ih}%
^{\mathcal{N}}(\mathcal{P})^{-}\right)  =\Delta V_{hh}^{\mathcal{N}%
}(\mathcal{P})^{+}\cup\Delta V_{hh}^{\mathcal{N}}(\mathcal{P})^{-}$. Moreover
it can be easily seen that $\Delta V_{ih}^{\mathcal{N}}(\mathcal{P}%
)^{+}\subset\Delta V_{hh}^{\mathcal{N}}(\mathcal{P})^{-}$ and $\Delta
V_{ih}^{\mathcal{N}}(\mathcal{P})^{-}\subset\Delta V_{hh}^{\mathcal{N}%
}(\mathcal{P})^{+}$ $\forall i\in I_{h}$. Indeed, any segment $s\in\Delta
V_{ih}^{\mathcal{N}}(\mathcal{P})^{+}$ is such that $s\in V_{i}^{\mathcal{N}%
}(\left\{  p_{1},\ldots,p_{h}+\varepsilon,\ldots,p_{m}\right\}  )$ and
$s\notin V_{i}^{\mathcal{N}}(\mathcal{P})$, and, for any infinitesimal
perturbation of $p_{h}$, it is possible only if $s\in V_{h}^{\mathcal{N}%
}(\mathcal{P})$ and $s\notin V_{h}^{\mathcal{N}}(\left\{  p_{1},\ldots
,p_{h}+\varepsilon,\ldots,p_{m}\right\}  )$, hence if $s\in\Delta
V_{hh}^{\mathcal{N}}(\mathcal{P})^{-}$. Therefore we have%
\begin{align*}
&  \frac{\partial}{\partial p_{h}}%
{\displaystyle\sum\limits_{i\in I_{h}}}
{\displaystyle\int\nolimits_{V_{i}^{\mathcal{N}}(\mathcal{P})}}
f\left(  \left\Vert q-p_{i}\right\Vert \right)  \phi\left(  q\right)  dq\\
&  =\lim_{\varepsilon\rightarrow0}\frac{1}{\left\Vert \varepsilon\right\Vert }%
{\displaystyle\sum\limits_{i\in I_{h}}}
\left(
{\displaystyle\int\nolimits_{\Delta V_{ih}^{\mathcal{N}}(\mathcal{P})^{+}}}
f\left(  \left\Vert q-p_{i}\right\Vert \right)  \phi\left(  q\right)  dq-%
{\displaystyle\int\nolimits_{\Delta V_{ih}^{\mathcal{N}}(\mathcal{P})^{-}}}
f\left(  \left\Vert q-p_{i}\right\Vert \right)  \phi\left(  q\right)
dq\right) \\
&  =\lim_{\varepsilon\rightarrow0}\frac{1}{\left\Vert \varepsilon\right\Vert
}\left(
{\displaystyle\int\nolimits_{\bigcup\nolimits_{i\in I_{h}}\Delta
V_{ih}^{\mathcal{N}}(\mathcal{P})^{+}}}
f\left(  \left\Vert q-p_{i}\right\Vert \right)  \phi\left(  q\right)  dq-%
{\displaystyle\int\nolimits_{\bigcup\nolimits_{i\in I_{h}}\Delta
V_{ih}^{\mathcal{N}}(\mathcal{P})^{-}}}
f\left(  \left\Vert q-p_{i}\right\Vert \right)  \phi\left(  q\right)
dq\right) \\
&  =\lim_{\varepsilon\rightarrow0}\frac{1}{\left\Vert \varepsilon\right\Vert
}\left(
{\displaystyle\int\nolimits_{\Delta V_{hh}^{\mathcal{N}}(\mathcal{P})^{-}}}
f\left(  \left\Vert q-p_{i}\right\Vert \right)  \phi\left(  q\right)  dq-%
{\displaystyle\int\nolimits_{\Delta V_{hh}^{\mathcal{N}}(\mathcal{P})^{+}}}
f\left(  \left\Vert q-p_{i}\right\Vert \right)  \phi\left(  q\right)
dq\right)  \text{.}%
\end{align*}
The conclusion follows from the fact that $\lim_{\varepsilon\rightarrow
0}\left(  \Delta V_{hh}^{\mathcal{N}}(\mathcal{P})^{+}\cup\Delta
V_{hh}^{\mathcal{N}}(\mathcal{P})^{-}\right)  =\partial V_{h}^{\mathcal{N}%
}(\mathcal{P})$ and $\forall q\in\partial V_{h}^{\mathcal{N}}(\mathcal{P})$
$\left\Vert q-p_{i}\right\Vert =\left\Vert q-p_{h}\right\Vert $. As concerns
the second-last term of (\ref{eq:firstterm_full}), recalling again that
$\lim_{\varepsilon\rightarrow0}\left(  \Delta V_{hh}^{\mathcal{N}}%
(\mathcal{P})^{+}\cup\Delta V_{hh}^{\mathcal{N}}(\mathcal{P})^{-}\right)
=\partial V_{h}^{\mathcal{N}}(\mathcal{P})$ and the Assumptions
\ref{Ass_fullnet}, we can write%
\begin{align*}
&  \lim_{\varepsilon\rightarrow0}\frac{1}{\left\Vert \varepsilon\right\Vert }%
{\displaystyle\int\nolimits_{\Delta V_{hh}^{\mathcal{N}}(\mathcal{P})^{+}}}
\left(  f\left(  \left\Vert q-\left(  p_{h}+\varepsilon\right)  \right\Vert
\right)  -f\left(  \left\Vert q-p_{h}\right\Vert \right)  \right)  \phi\left(
q\right)  dq\\
&  \leq\lim_{\varepsilon\rightarrow0}\frac{1}{\left\Vert \varepsilon
\right\Vert }%
{\displaystyle\int\nolimits_{\Delta V_{hh}^{\mathcal{N}}(\mathcal{P})^{+}}}
\left\Vert \frac{\partial f}{\partial x}\right\Vert _{\left[  0,\mathrm{diam}%
(\mathcal{N})\right]  }\left\Vert \varepsilon\right\Vert \left\Vert
\phi\right\Vert _{\left[  0,\mathrm{diam}(\mathcal{N})\right]  }dq\\
&  =\lim_{\varepsilon\rightarrow0}\left\Vert \frac{\partial f}{\partial
x}\right\Vert _{\left[  0,\mathrm{diam}(\mathcal{N})\right]  }\left\Vert
\phi\right\Vert _{\left[  0,\mathrm{diam}(\mathcal{N})\right]  }\mu\left(
\Delta V_{hh}^{\mathcal{N}}(\mathcal{P})^{+}\right)  =0\text{,}%
\end{align*}
where $\mathrm{diam}(\mathcal{N})\triangleq\max_{p,q\in\mathcal{N}}\left\Vert
q-p\right\Vert $. The same argument holds for the term with $\Delta
V_{hh}^{\mathcal{N}}(\mathcal{P})^{-}$, hence we have%
\begin{align*}
\frac{\partial\mathcal{H}}{\partial p_{h}}(\mathcal{P})  &  =\lim
_{\varepsilon\rightarrow0}\frac{1}{\left\Vert \varepsilon\right\Vert }%
{\displaystyle\int\nolimits_{V_{h}^{\mathcal{N}}(\mathcal{P})}}
\left(  f\left(  \left\Vert q-\left(  p_{h}+\varepsilon\right)  \right\Vert
\right)  -f\left(  \left\Vert q-p_{h}\right\Vert \right)  \right)  \phi\left(
q\right)  dq\\
&  =%
{\displaystyle\int\nolimits_{V_{h}^{\mathcal{N}}(\mathcal{P})}}
\frac{\partial}{\partial p_{h}}f\left(  \left\Vert q-p_{h}\right\Vert \right)
\phi\left(  q\right)  dq\\
&  =%
{\displaystyle\sum\limits_{k=1}^{M_{h}(\mathcal{P})}}
\mathcal{I}_{k}\text{,}%
\end{align*}
with $M_{h}(\mathcal{P})$ the number of segments in $V_{h}^{\mathcal{N}%
}(\mathcal{P})$ and $\mathcal{I}_{k}=%
{\displaystyle\int\nolimits_{s_{k}}}
\frac{\partial}{\partial p_{h}}f\left(  \left\Vert q-p_{h}\right\Vert \right)
\phi\left(  q\right)  dq$ and $s_{k}=[a_{k},b_{k}]\in V_{h}^{\mathcal{N}%
}(\mathcal{P})$. If we choose the parameterization $\gamma_{k}(t)=a_{k}%
+\left(  b_{k}-a_{k}\right)  t,$ $t\in\lbrack0,1]$ for $s_{k}$ we can apply
the Theorem \ref{Th:derseg} in appendix. Recall that in this case $\nu\left(
x,q\right)  =\left\Vert q-x\right\Vert $, hence the equation $\left\Vert
\gamma_{k}(t)-p_{h}\right\Vert -R_{\alpha}=0$ may have at most two zeros at
$t_{\alpha,1}^{k}$ and $t_{\alpha,2}^{k}$ $\forall\alpha\in\{1,\ldots,N\}$. It
is worth noting that assumptions \ref{Ass_fullnet} iii) and iv) play here the
same role of assumptions i) and ii) in Theorem \ref{Th:derseg}. Therefore,
from the definition of $f(\cdot)$, equation (\ref{eq:def_f}), we have the thesis.
\end{proof}

In order to define a gradient-like algorithm, also in this case, we must relax
Assumptions \ref{Ass_fullnet}. First of all, focus on the orthogonality
assumption. It has been introduced to avoid the presence of entire segments in
the boundary of a cell, because these configurations induce problems in the
definition of the gradient (they represent points on which the gradient may
assume different values). Even in this case we opt to use the lexicographic
rule in order to univocally assign a segment on the boundary to only one cell,
and, again, we consider a discrete-time dynamics for the gradient-like algorithm.

Using the lexicographic rule, we re-define the Voronoi cell as follows%
\begin{multline}
V_{i}^{\mathcal{N}}(\mathcal{P})=\{q\in\mathcal{N}\mid\left\Vert
q-p_{i}\right\Vert \leq\left\Vert q-p_{j}\right\Vert ~\forall p_{j}%
\in\mathcal{P}~\wedge\text{~}\nonumber\\
\left\Vert q-p_{i}\right\Vert <\left\Vert q-p_{j}\right\Vert \text{ if
}j<i\text{ w.r.t. the L.O.}\mathcal{\}}%
\end{multline}
and verify that the expression (\ref{eq:defH_Vor_red}) for $\mathcal{H}%
(\mathcal{P})$\ is still formally correct. We remove the orthogonality
hypothesis by adding to (\ref{eq:def_direcderNET_on_s}) an $\mathcal{I}_{k}$
term for each segment entirely included in the boundary of a Voronoi cell.
This fact does not change the expression (\ref{eq:def_direcderNET_on_s}),
since, with the new definition $V_{h}^{\mathcal{N}}(\mathcal{P})$,
$M_{h}(\mathcal{P})$ accounts now also for segments on the boundary.

The relaxation of the other assumptions would imply some discontinuities in
the integration domain induced by the discontinuities of the function $f$.
These discontinuities, without additional assumptions, would prevent us from
guaranteeing $\mathcal{H(P)}$ to be monotonically nondecreasing along the
evolution of $\mathcal{P}$ given by the gradient dynamics. Hence, we assume
now $f$ to be continuous and piecewise differentiable. Being $f$ continuous,
the second term in $\mathcal{I}_{k}$ in (\ref{eq:def_direcderNET_on_s}) is null.

As made in the previous section, the directional derivative must be univocally
defined on the vertices. To this aim, we use the expression
(\ref{eq:def_direcder_complete}) given in Definition
\ref{def:direcder_complete}, but with reference to the formula
(\ref{eq:def_direcderNET_on_s}) for the directional derivative in a point in
the interior of a segment. Using these definitions we can state the following theorem.

\begin{theorem}
Consider the following discrete-time evolution for the sensors' positions%
\begin{equation}
\mathcal{P}^{(k+1)}=\mathcal{P}^{(k)}+\delta_{k}\widetilde{D}\mathcal{H}%
(\mathcal{P}^{(k)})\text{,}%
\end{equation}
where the $h$-th component of $\widetilde{D}\mathcal{H}$ is given by
(\ref{eq:def_direcder_complete}) and $D_{w_{ij}}\mathcal{H(P)}[p_{h}]$ by
(\ref{eq:def_direcderNET_on_s}) and $\mathcal{H}:\mathcal{N}^{m}%
\rightarrow\mathbb{R}$ as in (\ref{eq:defH_Vor_red}). If $f(\cdot)$ has
locally bounded second derivatives, then, for suitable $\delta_{k}$,
$\mathcal{P}^{(k)}$ lies in a bounded set and

\begin{description}
\item[i)] $\mathcal{H}(\mathcal{P}^{(k)})$ is monotonically nondecreasing;

\item[ii)] $\mathcal{P}^{(k)}$ converges to the set of critical points of
$\mathcal{H}$.
\end{description}
\end{theorem}

\begin{proof}
As long as sensors' configurations not violating orthogonality assumption are
considered, the gradient is smooth and the proof is canonical. In the case of
discontinuity points, segments belonging to the boundary of a cell can change
allocation during sensors' motion. Hence, we can proceed as in the proof of
Theorem \ref{Th:discrevol_colR2} and \ref{Th:discrevol_colN} replacing
barycenters with segments. In particular, being $s_{h}$ a segment changing
allocation, equations (\ref{eq:Hu_contrib}) are now replaced by%
\begin{align*}
\mathcal{H}_{p_{i}}(\mathcal{P}^{(k+1)})  &  =\mathcal{H}_{p_{i}}%
^{u}(\mathcal{P}^{(k+1)})+\mathcal{I}_{h}^{i}\\
\mathcal{H}_{p_{j}}(\mathcal{P}^{(k+1)})  &  =\mathcal{H}_{p_{j}}%
^{u}(\mathcal{P}^{(k+1)})-\mathcal{I}_{h}^{j}\text{,}%
\end{align*}
with $\mathcal{I}_{k}^{l}=%
{\displaystyle\int\nolimits_{s_{h}}}
\frac{\partial}{\partial p_{l}}f\left(  \left\Vert q-p_{l}\right\Vert \right)
\phi\left(  q\right)  dq$. Again, due to the fact that any point $q\in s_{h}$
is such that $\left\Vert q-p_{i}^{(k+1)}\right\Vert \leq\left\Vert
q-p_{j}^{(k+1)}\right\Vert $\ and the monotonicity of $f$, we have
$\mathcal{I}_{h}^{i}\geq\mathcal{I}_{h}^{j}$, whereby the thesis follows as in
the proof of Theorem \ref{Th:discrevol_colR2}.
\end{proof}

\section{A case study\label{sec:simul}}

In this section we apply the proposed two-step optimization procedure to a
network representing a wing of the Amsterdam Schiphol airport. The network is
made up of $63$ vertices and $87$ segments and the density function $\phi$ is
the sum of $11$ Gaussian functions of the form $G(a,c_{x},c_{y},\sigma
_{x},\sigma_{y})=a\exp\left(  -\left(  \frac{x-c_{x}}{\sigma_{x}}\right)
^{2}-\left(  \frac{y-c_{y}}{\sigma_{y}}\right)  ^{2}\right)  $ with parameters
assuming the values given in table~\ref{tab:gauss}.

\begin{table}[ptb]
\centering%
\begin{tabular}
[c]{c|ccccccccccc}%
$a$ & $20$ & $20$ & $20$ & $10$ & $4$ & $20$ & $20$ & $20$ & $10$ & $10$ &
$4$\\
$c_{x}$ & $4.3$ & $5$ & $6$ & $3.5$ & $9$ & $12.5$ & $13.5$ & $15$ & $14.5$ &
$17$ & $20$\\
$c_{y}$ & $2.3$ & $4$ & $5.5$ & $5$ & $8.5$ & $8.5$ & $7.2$ & $6.2$ & $10.5$ &
$9$ & $7$\\
$\sigma_{x}$ & $1.5$ & $1.5$ & $1.5$ & $2$ & $4$ & $1.5$ & $1.5$ & $1.5$ & $2$
& $2$ & $4$\\
$\sigma_{y}$ & $1.5$ & $1.5$ & $1.5$ & $2$ & $4$ & $1.5$ & $1.5$ & $1.5$ & $2$
& $2$ & $2$\\\hline
\end{tabular}
\caption{Parameter values for the $11$ Gaussian functions $G(a,c_{x}%
,c_{y},\sigma_{x},\sigma_{y})=a\exp\left(  -\left(  \frac{x-c_{x}}{\sigma_{x}%
}\right)  ^{2}-\left(  \frac{y-c_{y}}{\sigma_{y}}\right)  ^{2}\right)  $
making up the density function $\phi$.}%
\label{tab:gauss}%
\end{table}The network and a contour plot of the density function are shown in
fig.~\ref{fig:SchipholNet} (darker colors denote preferential areas). For sake
of providing a clear graphical representation, the density function shown here
is defined on $\mathbb{R}^{2}$, but the one used in all simulations is
restricted to the network.\begin{figure}[h]
\begin{center}
\includegraphics[width=0.8\columnwidth]{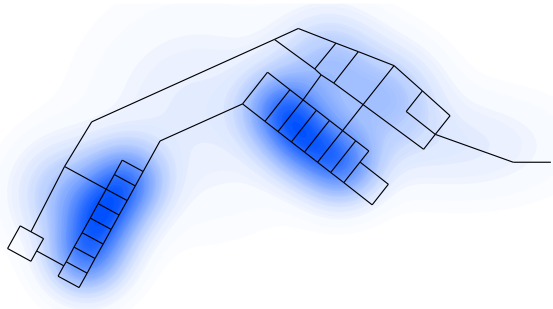}
\end{center}
\caption{Network representing a wing of the Amsterdam Schiphol airport and
contour plot of the density function $\phi$ used in simulations (darker colors
denote preferential areas).}%
\label{fig:SchipholNet}%
\end{figure}

The $50$ sensors to be deployed have performance function $f(x)=\frac{1}%
{2}\left(  1-\tanh\left(  \frac{x-\frac{R}{2}}{\frac{R}{6}}\right)  \right)
$, where $R$ is a parameter considered as variable in the first step and as
fixed in the second step of the optimization. Even if the previous function
describes sensors with an infinite sensing radius, they will be represented as
shaded circles of radius $\frac{7}{8}R$ to emphasize that the performance
function assumes values lesser than $0.01$ for larger distances.

\subsection{First step}

The first step of the optimization is performed on a collapsed network with
collapsing factor $r=0.3$ (see the small dots along the grey network in
fig.~\ref{fig:SchipholCP}-a,-c)). Sensors are grouped in $10$ clusters of $5$
elements each, and each cluster is represented as a single sensor. Clusters
set initially $R=10$ and decrease linearly its value up to $1$ during the
simulation (compare fig.~\ref{fig:SchipholCP}-a) with
fig.~\ref{fig:SchipholCP}-c)). As apparent by the flows in
fig.~\ref{fig:SchipholCP}-b), clusters are allowed to move in $\mathbb{R}^{2}$.

\begin{figure}[h]
\begin{center}%
\begin{tabular}
[c]{c}%
\begin{tabular}
[c]{c}%
\includegraphics[width=0.8\columnwidth]{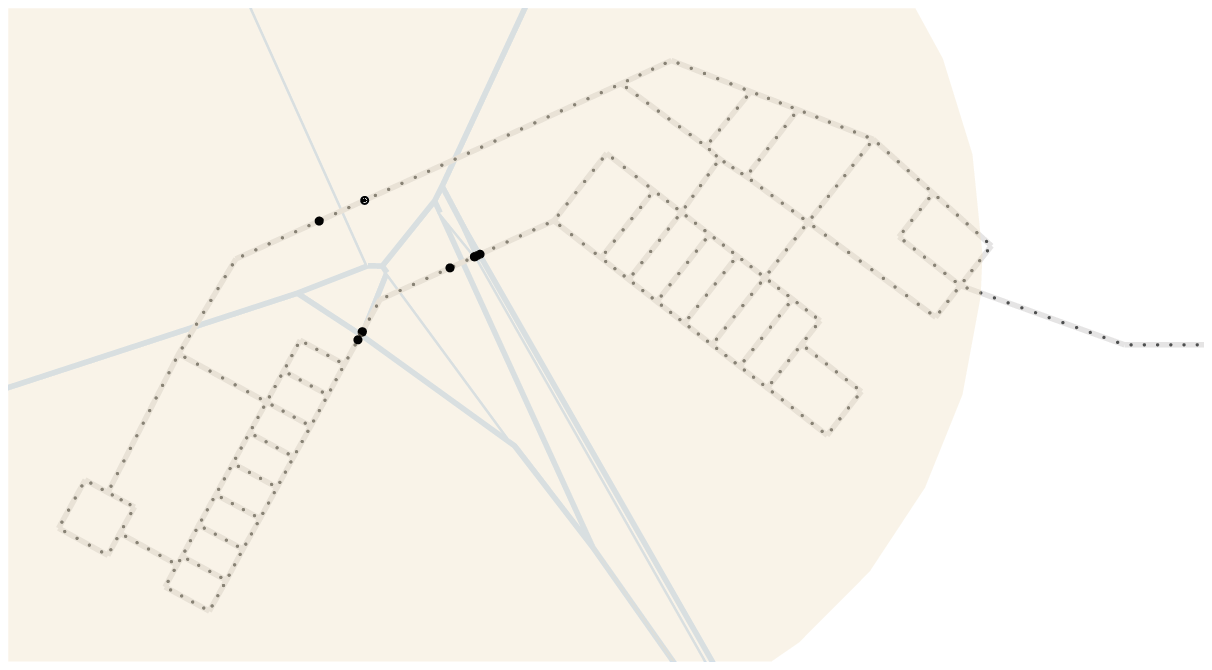}\\
a)
\end{tabular}
\\%
\begin{tabular}
[c]{c}%
\includegraphics[width=0.8\columnwidth]{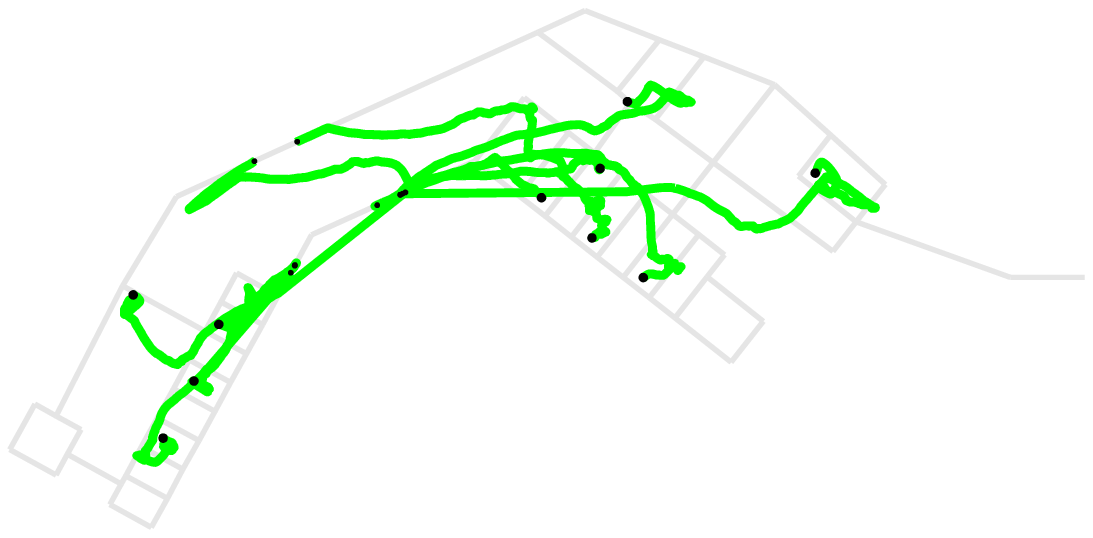}\\
b)
\end{tabular}
\\%
\begin{tabular}
[c]{c}%
\includegraphics[width=0.8\columnwidth]{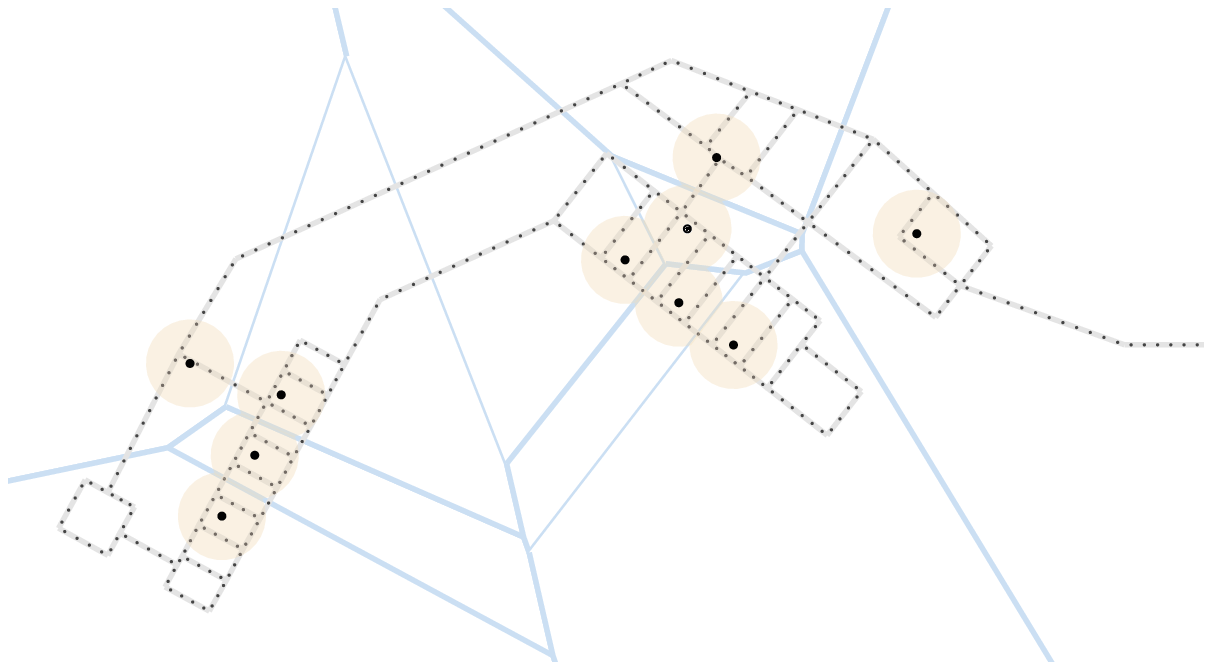}\\
c)
\end{tabular}
\end{tabular}
\end{center}
\caption{First step of the optimization procedure for the deployment over a
collapsed network of $50$ sensors grouped in $10$ clusters: a) initial
positions of clusters (sensing radius $R=10$) and related Voronoi partition;
b) gradient ascent flows (clusters allowed to move in $\mathbb{R}^{2}$); c)
final positions, not necessarily belonging to the network, (sensing radius
$R=1$) and related Voronoi partition. }%
\label{fig:SchipholCP}%
\end{figure}It is important to recall that, both the variation of the sensing
radius and the unconstrained motion of sensors are allowed in the first step
as it is performed off-line. This step makes use of the algorithm described in
section~\ref{sec:SensinR2} and is thought to provide a good starting point for
the second step. However, if sensors are initially located on the network as
in fig.~\ref{fig:SchipholCP}-a), they can execute the first step
independently, using partial or rough information of the environment, without
moving, and then plan a route on the network to reach the previously computed
final positions. Since final positions can be not on the network (see
fig.~\ref{fig:SchipholCP}-c)), they must be projected on it to be reachable.
Anyway, this projection has to be performed before the second step to provide
a valid starting point.

\subsection{Second step}

The second step considers a full network with sensors having fixed radius
$R=1$ and initially deployed in the positions shown in
fig.~\ref{fig:SchipholNN}-a). Such positions are obtained by spreading
randomly $5$ sensors close to each cluster center and projecting them on the
closest segments of the network. Sensors now can take real measures from the
environment and perform the optimization on-line, while moving, according to
the algorithm described in section~\ref{sec:FullNet}. They are constrained to
move on the network as shown in fig.~\ref{fig:SchipholNN}-b). Final positions
(see fig.~\ref{fig:SchipholNN}-c)) show how sensors, originally clustered,
diffused to better cover preferential areas (see fig~\ref{fig:SchipholNet}).
\begin{figure}[h]
\begin{center}%
\begin{tabular}
[c]{c}%
\begin{tabular}
[c]{c}%
\includegraphics[width=0.8\columnwidth]{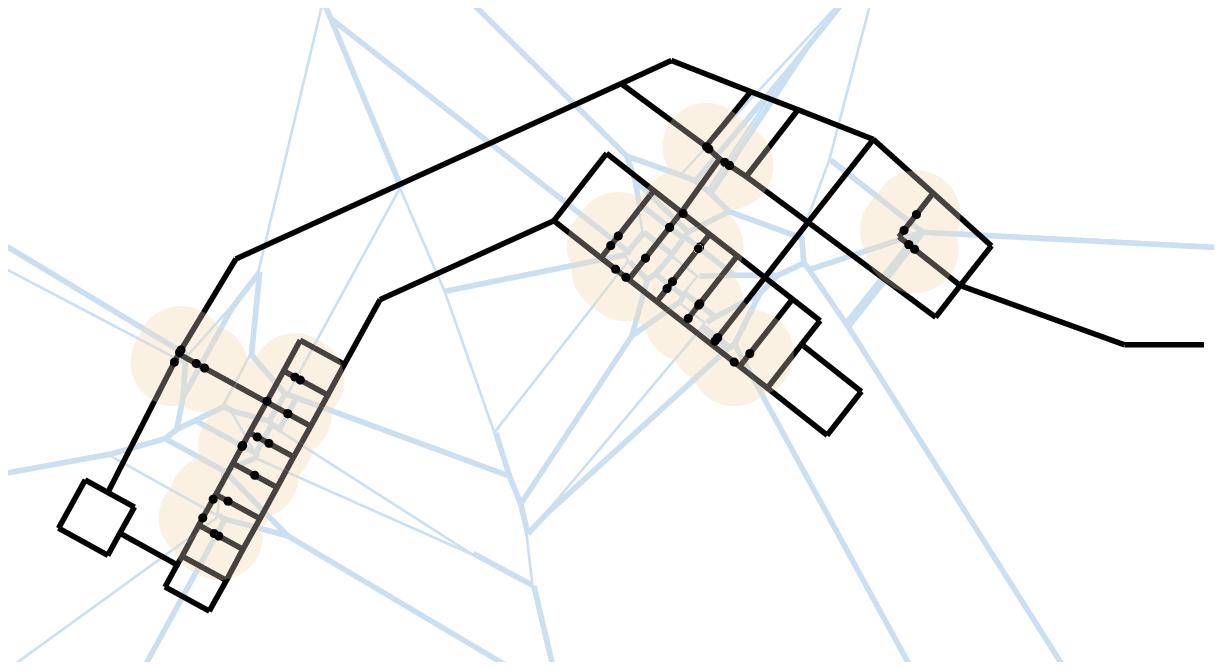}\\
a)
\end{tabular}
\\%
\begin{tabular}
[c]{c}%
\includegraphics[width=0.8\columnwidth]{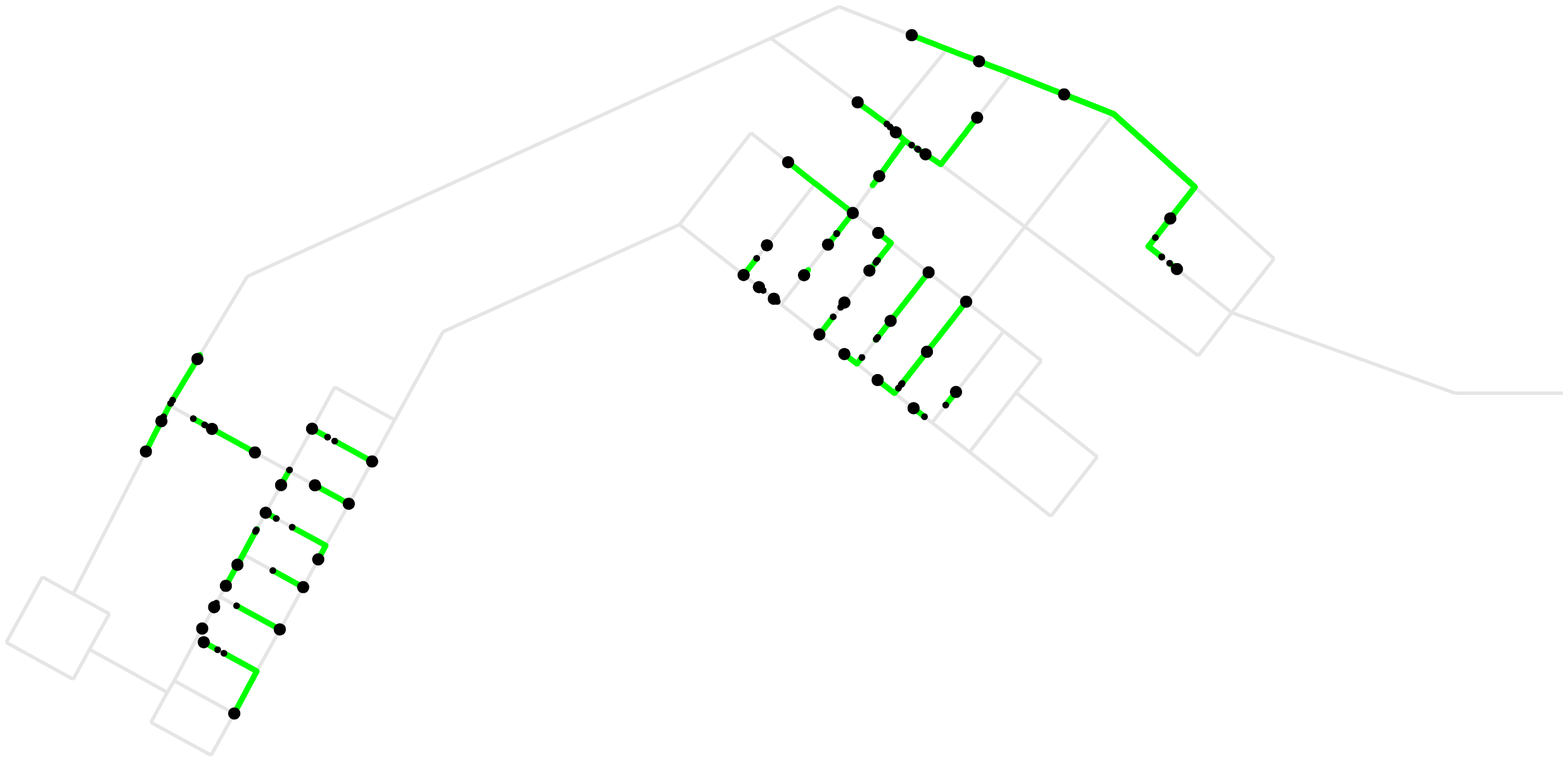}\\
b)
\end{tabular}
\\%
\begin{tabular}
[c]{c}%
\includegraphics[width=0.8\columnwidth]{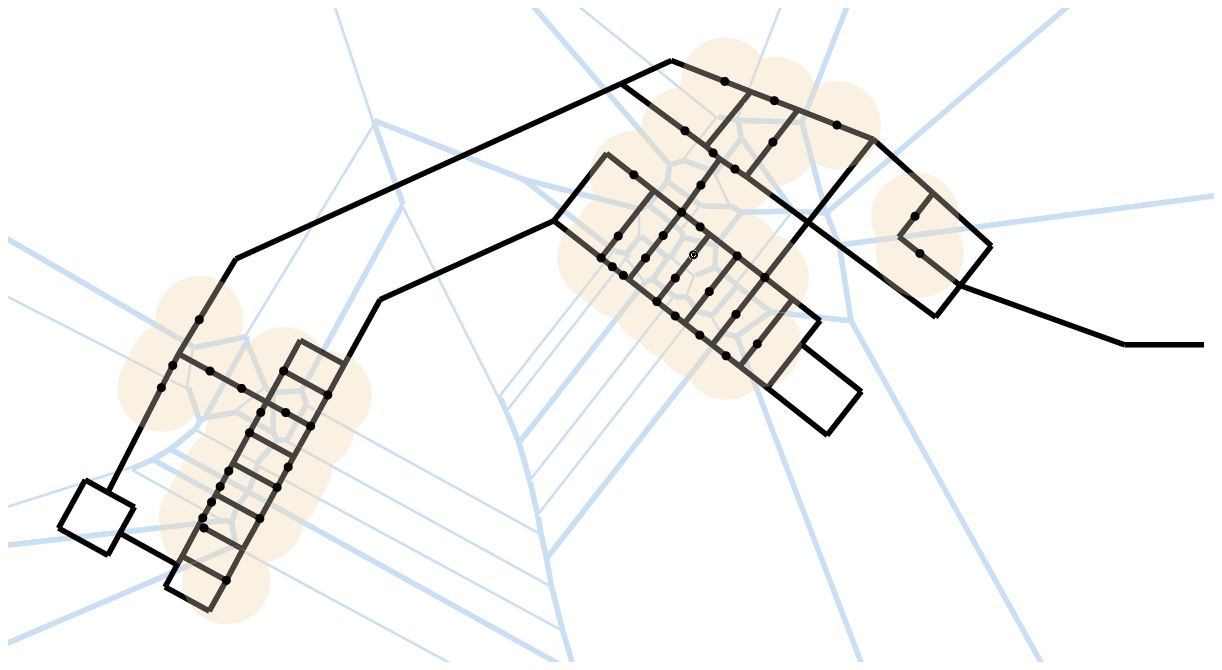}\\
c)
\end{tabular}
\end{tabular}
\end{center}
\caption{Second step of the optimization procedure for the deployment of $50$
sensors with $R=1$ over a full network: a) initial positions obtained by
spreading randomly $5$ sensors close to each cluster position found in the
first step (see fig.~\ref{fig:SchipholCP}-c)) and then projecting them on the
network, related Voronoi partition; b) gradient ascent flows (sensors
constrained to move on the network); c) final positions and related Voronoi
partition. }%
\label{fig:SchipholNN}%
\end{figure}

\section{Conclusions and Future Works\label{sec:conclusions}}

This paper focused on the problem of optimally deploying sensors in an
environment modeled as a network. An optimization problem for the allocation
of omnidirectional sensors with potentially limited sensing radius has been
formulated. A novel two-step optimization procedure based on a discrete-time
gradient ascent algorithm has been presented. In order the algorithm to not
get stuck early in one of the many local minima, in the first (off-line) step,
sensors are allowed to move in the plane. Moreover, a reduced model of the
environment, called collapsed network, and sensors' clustering, are used to
speed up the first optimization. The positions found in the first step are
then projected on the network and used in the second (on-line) finer
optimization, where sensors are constrained to move only on the network.

The proposed procedure can be used to solve both static and dynamic deployment
problems and the first step alone can provide solutions to location-allocation
problems involving facilities located in the interior of the network.

A main future research direction will consider the integration of classical
Operative Research methods with the present gradient algorithm. In particular
the first optimization could be addressed by adapting methods and heuristics
developed for the solution of the multisource Weber problem
(\cite{BrimbergOR00}). The aim is to build an overall global optimization
technique to solve location-allocation problems of large dimensions with many
facilities. Moreover, future research, more related to deployment problems,
will consider other sensor's models such as those with limited sensing cone.

\appendix{}

\section{\label{AppB}}

\begin{theorem}
\label{Th:derseg}Let $\varphi:\mathbb{R}^{2}\times\mathbb{R}_{+}%
\rightarrow\mathbb{R}$ be a smooth function w.r.t. its second argument and a
non-increasing, piecewise differentiable map with a finite number of bounded
discontinuities at $R_{1},\ldots,R_{N}\in\mathbb{R}_{+}$, $R_{1}<\ldots<R_{N}%
$, w.r.t. its first argument. Let $\nu:\mathbb{R}^{2}\times\mathbb{R}%
^{2}\rightarrow\overline{\mathbb{R}}_{+}$ be a continuously differentiable map
w.r.t. both its arguments. Let $s=[a,b]\subset\mathbb{R}^{2}$ be a segment and
assume that

\begin{description}
\item[i)] $\nu(\bar{x},a),\nu(\bar{x},b)\notin\left\{  R_{1},\ldots
,R_{N}\right\}  $;

\item[ii)] $\forall q\in\lbrack a,b]$, if $\nabla\nu(\bar{x},q)\cdot
(b-a)=0\Rightarrow\nu(\bar{x},q)\notin\left\{  R_{1},\ldots,R_{N}\right\}  $,
\end{description}

then%
\begin{align*}
&  \frac{d}{dx}\left.
{\displaystyle\int\nolimits_{s}}
\varphi\left(  \nu\left(  x,q\right)  ,q\right)  dq\right\vert _{x=\bar{x}}\\
&  =%
{\displaystyle\int\nolimits_{[0,1]\backslash\left\{  t_{1,1},\ldots
,t_{1,k_{1}},\ldots,t_{N,1},\ldots,t_{N,k_{N}}\right\}  }}
\frac{\partial}{\partial\nu}\varphi(\nu(\bar{x},\gamma(t)),\gamma(t))\left.
\frac{\partial\nu(x,\gamma(t))}{\partial x}\right\vert _{x=\bar{x}}\left\Vert
b-a\right\Vert dt\\
&  +%
{\displaystyle\sum\limits_{i=1}^{N}}
{\displaystyle\sum\limits_{j=1}^{k_{i}}}
\left(  \varphi(R_{i}^{+},\gamma(t_{i,j}))-\varphi(R_{i}^{-},\gamma
(t_{i,j}))\right)  \left.  \frac{\partial\nu(x,\gamma(t))}{\partial
x}\right\vert _{\substack{x=\bar{x}\\t=t_{i,j}}}\left\Vert b-a\right\Vert
\text{,}%
\end{align*}
where $\gamma(t)=a+\left(  b-a\right)  t,$ $t\in\lbrack0,1]$ is a
parameterization for $s$ and $t_{i,j}\in\lbrack0,1],~j\in\{1,\ldots,k_{i}\}$
are the zeros of $\nu(\bar{x},\gamma(t))-R_{i}=0$.
\end{theorem}

\begin{proof}
By using the Dirac's delta formalism we have%
\begin{align*}
&  \frac{d}{dx}\left.
{\displaystyle\int\nolimits_{s}}
\varphi\left(  \nu\left(  x,q\right)  ,q\right)  dq\right\vert _{x=\bar{x}}\\
&  =\left.
{\displaystyle\int\nolimits_{s}}
\frac{\partial}{\partial x}\varphi(\nu(x,q),q)dq\right\vert _{x=\bar{x}}\\
&  =%
{\displaystyle\int\nolimits_{s}}
\left.  \frac{\partial}{\partial\nu}\varphi(\nu(x,q),q)\right\vert _{\nu
\notin\left\{  R_{1},\ldots,R_{N}\right\}  }\left.  \frac{\partial\nu
(x,q)}{\partial x}\right\vert _{x=\bar{x}}dq\\
&  +%
{\displaystyle\int\nolimits_{s}}
{\displaystyle\sum\limits_{i=1}^{N}}
\left(  \varphi(R_{i}^{+},q)-\varphi(R_{i}^{-},q)\right)  \delta\left(
\nu-R_{i}\right)  \left.  \frac{\partial\nu(x,q)}{\partial x}\right\vert
_{\nu=R_{i}}dq\text{.}%
\end{align*}
With the chosen parameterization of $s$, the equation $\nu(\bar{x}%
,\gamma(t))-R_{i}=0$ may have $k_{i}$ zeros $t_{i,j}\in\lbrack0,1],~j\in
\{1,\ldots,k_{i}\}$. Thanks to i) and ii) the set of zeros $t_{i,j}$ does not
change cardinality for $x$\ near $\bar{x}$, thus, $t_{i,j}$\ depends smoothly
on $x$. Recalling that for every arc $\gamma:[l,u]\rightarrow\Gamma$ we have%
\[%
{\displaystyle\int\nolimits_{\Gamma}}
f(x)dx=%
{\displaystyle\int\nolimits_{l}^{u}}
f(\gamma(t))\left\Vert \dot{\gamma}(t)\right\Vert dt\text{,}%
\]
and that for the special case of $s$ $\left\Vert \dot{\gamma}(t)\right\Vert
=\left\Vert b-a\right\Vert $, the derivative becomes%
\begin{align*}
&  \frac{d}{dx}\left.
{\displaystyle\int\nolimits_{s}}
\varphi\left(  \nu\left(  x,q\right)  ,q\right)  dq\right\vert _{x=\bar{x}}\\
&  =%
{\displaystyle\int\nolimits_{[0,1]\backslash\left\{  t_{1,1},\ldots
,t_{1,k_{1}},\ldots,t_{N,1},\ldots,t_{N,k_{N}}\right\}  }}
\frac{\partial}{\partial\nu}\varphi(\nu(\bar{x},\gamma(t)),\gamma(t))\left.
\frac{\partial\nu(x,\gamma(t))}{\partial x}\right\vert _{x=\bar{x}}\left\Vert
b-a\right\Vert dt\\
&  +%
{\displaystyle\int\nolimits_{[0,1]}}
{\displaystyle\sum\limits_{i=1}^{N}}
\left(  \varphi(R_{i}^{+},\gamma(t))-\varphi(R_{i}^{-},\gamma(t))\right)
\delta\left(  \nu-R_{i}\right)  \left.  \frac{\partial\nu(x,\gamma
(t))}{\partial x}\right\vert _{\nu=R_{i}}\left\Vert b-a\right\Vert dt\text{,}%
\end{align*}
from which, using the property $%
{\displaystyle\int\nolimits_{l}^{u}}
f(x)\delta\left(  x-\bar{x}\right)  dx=f(\bar{x})$ for $\bar{x}\in\lbrack
l,u]$, we have the thesis.
\end{proof}

\bibliographystyle{plain}
\bibliography{acompat,Deployment}

\end{document}